\documentclass[12pt,reqno]{amsart}
\usepackage[margin=1in]{geometry}
\usepackage{amsmath,amssymb,amsthm,graphicx,amsxtra, setspace}
\usepackage[utf8]{inputenc}
\usepackage{mathrsfs}
\usepackage{hyperref}
\usepackage{upgreek}
\usepackage{mathtools}
\usepackage{xcolor}
\allowdisplaybreaks

\usepackage[pagewise]{lineno}

\usepackage{graphicx,eurosym}
\usepackage{hyperref}
\usepackage{mathtools}

\usepackage[cyr]{aeguill}

\colorlet{darkblue}{blue!50!black}

\hypersetup{
	colorlinks,%
	citecolor=blue,%
	filecolor=red,%
	linkcolor=red,%
	urlcolor=blue,%
	pdfnewwindow=true,%
	pdfstartview={FitH}
}

\newtheorem{theorem}{Theorem}[section]
\newtheorem{lemma}[theorem]{Lemma}
\newtheorem{proposition}[theorem]{Proposition}

\newtheorem{definition}[theorem]{Definition}

\newtheorem{remark}[theorem]{Remark}

\newtheorem{hypothesis}[theorem]{Hypothesis}

\DeclareMathOperator*{\esssup}{ess\,sup}

\let\originalleft\left
\let\originalright\right
\renewcommand{\left}{\mathopen{}\mathclose\bgroup\originalleft}
\renewcommand{\right}{\aftergroup\egroup\originalright}

\renewcommand{\d}{\/\mathrm{d}\/}

\def\w{\textbf{W}^{\varepsilon}_{{\theta}^{\varepsilon}}}

\def\e{\varepsilon}

\def\L{\mathbb{L}}
\def\A{\mathrm{A}}

\def\F{\mathrm{F}}
\def\C{\mathrm{C}}
\def\f{\boldsymbol{f}}

\def\B{\mathrm{B}}
\def\D{\mathrm{D}}
\def\y{\boldsymbol{y}}

\def\E{\mathbb{E}}

\def\x{\boldsymbol{x}}
\def\k{\boldsymbol{k}}

\def\h{\boldsymbol{h}}

\def\v{\boldsymbol{v}}
\def\V{\mathbb{v}}
\def\w{\boldsymbol{w}}
\def\W{\mathrm{W}}
\def\G{\mathrm{G}}

\def\M{\mathrm{M}}
\def\N{\mathbb{N}}

\def\V{\mathbb{V}}
\def\wi{\widetilde}

\def\u{\mathrm{U}}

\def\u{\boldsymbol{u}}
\def\H{\mathbb{H}}

\newcommand{\R}{\mathbb{R}}

\renewcommand{\d}{\/\mathrm{d}\/}

\numberwithin{equation}{section}

\newcommand{\Addresses}{{% additional braces for segregating \footnotesize
		\footnote{
			\noindent \textsuperscript{1,2}Department of Mathematics, Indian Institute of Technology Roorkee-IIT Roorkee,
			Haridwar Highway, Roorkee, Uttarakhand 247667, INDIA.\par\nopagebreak
			\noindent  \textit{e-mail:} \texttt{Manil T. Mohan: maniltmohan@ma.iitr.ac.in, maniltmohan@gmail.com.}

			\noindent \textsuperscript{*}Corresponding author.
			
			\textit{Key words:} Stochastic convective Brinkman-Forchheimer extended Darcy equations, Approximations, Gaussian noise,  jump noise. 
			
			Mathematics Subject Classification (2020): Primary 60H15; Secondary 35R60, 35Q30, 76D05.

}}}
\begin{document}	%	\linenumbers
	
	\title[Approximations of 2D and 3D Stochastic CBFeD Equations]{Approximations of 2D and 3D Stochastic Convective Brinkman-Forchheimer Extended Darcy Equations
		\Addresses}

	\author[M. T. Mohan]
	{Manil T. Mohan\textsuperscript{*}}

	\maketitle

		\begin{abstract}
			\noindent  In this article, we consider two- and three- dimensional stochastic convective Brinkman-Forchheimer extended Darcy (CBFeD) equations 
				\begin{equation*}
				\frac{\partial \boldsymbol{u}}{\partial t}-\mu \Delta\boldsymbol{u}+(\boldsymbol{u}\cdot\nabla)\boldsymbol{u}+\alpha|\boldsymbol{u}|^{q-1}\boldsymbol{u}+\beta|\boldsymbol{u}|^{r-1}\boldsymbol{u}+\nabla p=\boldsymbol{f},\ 
				\nabla\cdot\boldsymbol{u}=0, 
			\end{equation*}
		on a torus, where $\mu,\beta>0$, $\alpha\in\mathbb{R}$, $r\in[1,\infty)$ and $q\in[1,r)$. 	The goal  is to show that the solutions of 2D and 3D stochastic CBFeD  equations driven	by Brownian motion can be approximated by 2D and 3D stochastic CBFeD equations forced by pure jump	noise/random kicks on on the state space $\mathrm{D}([0,T];\mathbb{H})$. The results are established  for $d=2,r\in[1,\infty)$  and $d=3,r\in[3,\infty)$ with $2\beta\mu\geq 1$ for $d=r=3,$ and by using less regular assumptions on the noise coefficient.
		\end{abstract}

	\section{Introduction}\label{Sec1}
	A great deal of research is being conducted to mathematically model and analyze nonlinear flows and transport processes through a porous medium. Many models of porous media are based on Darcy's law and Darcy-Forchheimer's law (cf. \cite{PAM}).  Darcy's law is an equation that describes the flow of a fluid through a porous medium, derived for slow linearly viscous flows where momentum transfer dissipation is negligible. Darcy's empirical flow model suggests that there is a  linear relationship  between flow rate and the pressure drop in a porous media,  that is, $\nabla p=-\frac{\nu}{\kappa}\u_f,$ where $\u_f$ is the Darcy velocity, $\kappa$ is the permeability of the porous medium, $\nu$  is the dynamic	viscosity of the fluid, and $p$ is the pressure.  In certain cases, nature does not abide by Darcy's law; instead, it follows Forchheimer law, which states that the relationship between the flow rate and pressure gradient is nonlinear at high velocities, and this nonlinearity increases with the increasing flow rate (cf. \cite{PAM}).  The Darcy-Forchheimer law states that $\nabla p=-\frac{\nu}{\kappa}\v_f-\gamma\rho_f|\v_f|^2\v_f$, where $\gamma>0$ is the  Forchheimer coefficient, $\v_f$ stands for the Forchheimer velocity and $\rho_f$ is the density.  Therefore, Forchheimer's law suggests that Darcy's law is still applicable, but with an additional nonlinear term to account for the increased pressure drop. In this work, we consider an approximation for the  stochastic counterpart of convective Brinkman-Forchheimer extended Darcy (CBFeD) model, which  is based on a Darcy-Forchheimer law.

	\subsection{The model and literature survey}
	Let $L>0$, $d=2,3$ and $\mathbb{T}^d=\big(\R/\mathrm{L}\mathbb{Z}\big)^{d}$, $L>0,$ $d\in\{2,3\}$  be a torus. The authors in \cite{PAM} introduced the following CBFeD  model:
	\begin{align}\label{1a}
		\left\{
		\begin{aligned}
		\frac{\partial \u}{\partial t}-\mu \Delta\u+(\u\cdot\nabla)\u+\alpha|\u|^{q-1}\u+\beta|\u|^{r-1}\u+\nabla p&=\f, \ \text{ in } \ \mathbb{T}^d\times(0,T), \\
		\nabla\cdot\u&=0, \ \text{ in } \ \mathbb{T}^d\times(0,T),
	\end{aligned}
\right.
\end{align}
	with the initial condition 
	\begin{align}\label{1c}
		\u(0)=\h, \ \text{ in } \ \mathbb{T}^d,
	\end{align}
	and $\u$ satisfies the periodic boundary conditions 
	\begin{align}\label{1d}
		\u(t,x+\mathrm{L}e_i)=\u(t,x),\ t\in[0,T],
	\end{align}
	for	every $x\in\R^d$ and $i=1,\ldots,d$, where	$\{e_1,\ldots,e_d\}$ is the canonical basis	of $\R^d$.   Here   $\u(x,t) :\mathbb{T}^d\times[0,T]\to \R^d$ denotes the velocity field  $p(x,t):\mathbb{T}^d\times[0,T]\to\R$ represents the pressure field  and  $\f(t,x) :\mathbb{T}^d\times[0,T]\to \R^d$ stands for an external forcing. For the uniqueness of pressure $p$, one can impose the condition $	\int_{\mathbb{T}^d}p(x,t)\d x=0 \ \text{ in } \ [0,T].$ The constants $\mu,\alpha,\beta$ denote the positive Brinkman  (effective viscosity), Darcy (permeability of porous medium for $q=1$) and Forchheimer (proportional to the porosity of the material) coefficients, respectively. It can be easily seen that for $\alpha=\beta=0$, one can obtain the classical $d$-dimensional Navier-Stokes equations (NSE).  In the system \eqref{1a}, the extra term $\alpha|\u|^{q-1}\u$ is  introduced to model a \emph{pumping},	when $\alpha<0$, by opposition to the \emph{damping} modeled through the term $\beta|\u|^{r-1}\u$  when $\alpha>0$.	For $q=1$ and $\alpha>0$, the system \eqref{1a} is known as convective Brinkman-Forchheimer (CBF)  equations. 
	
%	The  convective Brinkman-Forchheimer (CBF)  equations characterize the motion of incompressible fluid flows in a saturated porous medium. The CBF  model  is recognized to be more accurate when the flow velocity is too large for Darcy's law to be valid alone, and in addition, the porosity is not too small, so that  the term \emph{non-Darcy models} is used in the literature  for such types of fluid flow models (see \cite{PAM} for a discussion).  	

	 The exponent $r\in[1,\infty)$ is referred as the \emph{absorption exponent} and the case $r=3$ is known as \emph{critical exponent} and the case  $r>3$ is called  supercritical or fast growing nonlinearity (cf. \cite{KT2}). It has been established  in Proposition 1.1, \cite{KWH}  that the critical homogeneous CBF  equations have the same scaling as the  NSE only when the permeability coefficient $\alpha=0$ and no scale invariance property for other values of $\alpha$ and $r$. The \emph{tamed Navier-Stokes equations} were proposed by the authors of \cite{MRXZ}, featuring a dissipative term known as a \emph{taming function} to counterbalance the convective term in the classical 3D NSE. It can be inferred that the authors of \cite{MRXZ} have taken $r=3$, and $\mu=\beta=1$  in the first equation in \eqref{1a} with $\alpha=0$, and  they have obtained the existence of strong solution (in the analytic sense) (cf. Theorem 4.1, \cite{KWH} with $4\beta\mu\geq 1$).   For $d=2,r\in[1,\infty)$  and $d=3,r\in[3,\infty)$ ($4\beta\mu\geq 1$ for $d=r=3$), the global solvability results (the existence and uniqueness of weak and strong solutions) of deterministic CBF and CBFeD equations in bounded and periodic domains are available in the works \cite{SNA,CLF,KWH,KT2,PAM,MTM5}, etc. and the references therein. As in the case of 3D NSE, the existence of global  strong solutions for 3D CBF and CBFeD equations is not known for subcritical case $r\in[1,3)$ ($4\beta\mu<1$ for $r=3$).
	 
	 In this work, we consider the following 2D and 3D stochastic CBFeD equations  on a torus: 
	 	\begin{align}\label{1.5}
	 	\left\{
	 	\begin{aligned}
	 		\d\u&+\left[-\mu \Delta\u+(\u\cdot\nabla)\u+\alpha|\u|^{q-1}\u+\beta|\u|^{r-1}\u+\nabla p\right]\d t\\&=\F(\u)\d t+\sum\limits_{i=1}^m\sigma^i(\u)\d\W^i, \ \text{ in } \ \mathbb{T}^d\times(0,T), \\
	 		\nabla\cdot\u&=0, \ \text{ in } \ \mathbb{T}^d\times(0,T),\\
	 			\u(0)&=\h,\ \text{ in } \ \mathbb{T}^d,
	 	\end{aligned}
	 	\right.
	 \end{align}
 with the periodic condition given in \eqref{1d}. In \eqref{1.5}, $\W = \{(\W_1(t),\ldots,\W_m (t))\}_{t\geq 0}$  is an $m$-dimensional standard Brownian motion on a complete probability space $(\Omega,\mathscr{F},\{\mathscr{F}_t\}_{t\geq 0},\mathbb{P})$. The fluid is driven by external force $\F(\u)$ and the random noise $\sum\limits_{i=1}^m\sigma^i(\u(\cdot))\d\W^i(\cdot)$. 	
 
 The works \cite{HBAM,ZBGD,WLMR,MRXZ1}, etc. have established the existence and uniqueness of pathwise strong solution of the stochastic tamed NSE and related models (forced by Gaussian) in the whole space or on a torus.  The authors in  \cite{ZDRZ}  proved the existence and uniqueness of a strong solution to the stochastic 3D tamed NSE driven by multiplicative L\'evy noise, with periodic boundary conditions, through Galerkin's approximation and a type of local monotonicity of the coefficients. 	In \cite{LHGH1}, the existence of martingale solutions for stochastic 3D NSE with nonlinear damping forced by multiplicative Gaussian noise is obtained by employing a classical Faedo-Galerkin approximation and compactness method. By exploiting a monotonicity property of the linear and nonlinear operators as well as a stochastic generalization of the Minty-Browder technique, the author in \cite{MTM4,MTM6} established the existence and uniqueness of \emph{a global  strong solution} \begin{align}\label{15}\u\in\D([0,T];\H)\cap\mathrm{L}^2(0,T;\V)\cap\mathrm{L}^{r+1}(0,T;\wi\L^{r+1}), \ \mathbb{P}\text{-a.s.,}\end{align} satisfying the energy equality (It\^o's formula) for stochastic CBF equations (in bounded and periodic domains) driven by multiplicative Gaussian and pure jump noise, respectively, for $d=2,3$ and $r\in[3,\infty)$ ($2\beta\mu\geq 1$ for $d=r=3$). Under suitable assumptions on the initial data and noise coefficients, they have also showed the regularity result \begin{align}\label{16}\u\in\D([0,T];\V)\cap\mathrm{L}^2(0,T;\D(\A))\cap\mathrm{L}^{r+1}(0,T;\wi\L^{p(r+1)}),\  \mathbb{P}\text{-a.s.,}\end{align} where $p\in[2,\infty)$ for $d=2$ and $p=3$ for $d=3$. The existence and uniqueness of local and global pathwise mild solutions for stochastic CBF equations perturbed by additive L\'evy noise in $\R^d$, $d=2,3$, is established in \cite{MTM9} via the contraction mapping principle. In \cite{MTM10}, the existence of a weak martingale solution for 2D and 3D stochastic CBF equations perturbed by L\'evy noise is proved, employing the classical Faedo-Galerkin approximation, a compactness method, and a version of the Skorokhod embedding theorem for nonmetric spaces (for $d=2,3$ and $r\in[1,\infty)$). One can employ similar methods as in \cite{MTM4,MTM6} to obtain the global solvability results for the 2D and 3D stochastic CBFeD equations \eqref{1.5} with regularity given in \eqref{15} and \eqref{16} also. 
	
	\subsection{Aims, novelties, difficulties and approaches}
The major goal of this article is to study the approximations of stochastic CBFeD equations  in \eqref{1.5} by stochastic CBFeD equations  forced by Poisson random measures. One of the key motivations for studying this problem is to way a path for the numerical simulations of stochastic CBFeD equations driven by pure jump noise. The authors of \cite{GDNT} initiated work in the direction of approximating solutions of stochastic partial differential equations (SPDEs) including stochastic Burgers equation driven by Brownian motion by SPDEs forced by pure jump noise/random kicks. The works \cite{SSTZ,LHCSJX,XPRZ}, etc., extended this results to  stochastic 2D Navier-Stokes equations,  stochastic 3D Navier-Stokes equations with damping  and stochastic 3D tamed Navier-Stokes equations, respectively. 

We point out here that the approximation results obtained in the work \cite{XPRZ} for stochastic 3D tamed Navier-Stokes equations (a special case of 3D critical CBF equations with $\mu=\beta=1$) is in the state space $\D([0,T];\V)$ with an additional $\H^2$ regularity assumption on the noise coefficient (\cite[Hypothesis H5]{XPRZ}).  It should also be noted that the approximation results derived  in \cite{LHCSJX} for  stochastic 3D Navier-Stokes equations with damping  hold true for the system \eqref{1.5} with $\alpha=0$, $3<r<5$ and $4\beta\mu\geq 1$ for $r=3$ in the state space $\D([0,T];\H)$ and $\D([0,T];\V)$ with an additional $\H^2$ regularity assumption on the noise coefficient (\cite[Assumption 2.2]{LHCSJX}). In this work, using  less regular assumptions  (see Hypothesis \ref{hyp4} below) on the noise coefficient compared to \cite{LHCSJX,XPRZ} (cf. \cite[Hypothesis H.5]{SSTZ} for 2D stochastic NSE), we prove the approximation results for stochastic 2D and 3D CBFeD equations in $\D([0,T];\H)$ for $d=2$ with $r\in[1,\infty)$ and $d=3$ with $r\in[3,\infty)$ ($2\beta\mu>1 $ for $d=r=3$).  

We follow the works \cite{GDNT,SSTZ,LHCSJX,XPRZ}, etc. to obtain the main  result of this paper. In order to prove  approximations of stochastic CBFeD equations by pure jump type stochastic CBFeD equations in $\D([0,T];\H)$, we first establish the usual energy estimates (Lemmas \ref{lem3.6} and \ref{lem3.7}) under mild assumptions on the initial data ($\h\in\H$) and noise coefficient (Hypothesis \ref{hyp1} and \ref{hyp2}).   The difficulty lies in establishing the tightness of the approximating equations in the state space $\D([0,T];\H)$. Assuming that the initial data has higher regularity ($\h\in\V$), and that the coefficients of the jump noise take values in a more regular space (Hypothesis \ref{hyp4}), we  overcome this difficulty by deriving a uniform estimate of the stronger norm of the approximating solutions (Lemma \ref{lem3.8}). It should be noted that Hypothesis \ref{hyp4} is weaker than the regularity assumptions on the noise coefficients  made in \cite[Hypothesis H5]{XPRZ} and  \cite[Assumption 2.2]{LHCSJX}.  By using Aldou's criterion, we can prove the tightness of the approximating equations in $\D([0,T];\H)$ with these uniform estimates (Proposition \ref{prop3.9}). We demonstrate, through a martingale characterization and \cite[Lemma 15]{SSS}, that the limit of the solutions of approximating equations is the solution of the 2D and 3D stochastic CBFeD equations driven by Brownian motion \eqref{1.5}  for $d=2,r\in[1,\infty)$  and $d=3,r\in[3,\infty)$ with $2\beta\mu\geq 1$ for $d=r=3,$ (Theorem \ref{thm310}).  Finally, by using finite dimensional approximations and establishing uniform convergence in probability of the approximating solutions, we can remove the regularity restrictions on the coefficients and the initial condition (Theorem \ref{thm3.11}).

\subsection{Organization of the paper}
The rest of the paper is organized as follows. In the following section, we present the functional spaces and operators required to achieve the primary outcomes of this study. The main result of approximations of stochastic CBFeD equations by pure jump type stochastic CBFeD equations in $\D([0,T];\H)$ is provided in section \ref{Sec3} (Theorem \ref{thm3.11}) by using some auxiliary results on suitable energy estimates (Lemmas \ref{lem3.6}-\ref{lem3.8}), tightness property by Aldou's criterion (Proposition \ref{prop3.9}) and weak convergence in a stronger topology (Theorem \ref{thm310}).

\section{Functional Setting}\label{Sec2}\setcounter{equation}{0}
This section is devoted for providing the necessary function spaces needed for further analysis of this work. We consider the problem \eqref{1a}-\eqref{1d}  on a $d$-dimensional torus $\mathbb{T}^d$ with the periodic boundary conditions and zero-mean value constraint for the functions, that is, $\int_{\mathbb{T}^d}\u(x)\d x=\textbf{0}$.

\subsection{Function spaces} Let  $\mathring{\C}_p^{\infty}(\mathbb{T}^d;\R^d)$ denote the space of all infinite times differentiable  functions ($\mathbb{R}^d$-valued) such that $\int_{\mathbb{T}^d}\u(x)\d x=\textbf{0}$ and $\u(x+\mathrm{L}e_i)=\u(x),$ for	every $x\in\R^d$ and $i=1,\ldots,d$, where	$\{e_1,\ldots,e_d\}$ is the canonical basis	of $\R^d$. The Sobolev space  $\mathring{\H}_p^k(\mathbb{T}^d):=\mathring{\mathrm{H}}_p^k(\mathbb{T}^d;\mathbb{R}^d)$ is the completion of $\mathring{\C}_p^{\infty}(\mathbb{T}^d;\R^d)$  with respect to the $\H^s$ norm $\|\u\|_{\mathring{\H}^s_p}:=\left(\sum\limits_{0\leq|\alpha|\leq s}\|\D^{\alpha}\u\|_{\mathbb{L}^2(\mathbb{T}^d)}^2\right)^{\frac{1}{2}}.$ The Sobolev space of periodic functions with zero mean $\mathring{\H}_p^k(\mathbb{T}^d)$ is the same as (Proposition 5.39, \cite{Robinson1}) $$\left\{\u:\u=\sum\limits_{\k\in\mathbb{Z}^d}\u_{\k} e^{2\pi  i\k\cdot\x /  L},\u_0=\mathbf{0},\ \bar{\u}_{\k}=\u_{-\k},\ \|\u\|_{\mathring{\H}^s_f}:=\sum\limits_{k\in\mathbb{Z}^d}|\k|^{2s}|\u_{\k}|^2<\infty\right\}.$$ From Proposition 5.38, \cite{Robinson1}, we infer that the norms $\|\cdot\|_{\mathring{\H}^s_p}$ and $\|\cdot\|_{\mathring{\H}^s_f}$ are equivalent. Let us define 
\begin{align*} 
	\mathcal{V}&:=\{\u\in\mathring{\C}_p^{\infty}(\mathbb{T}^d;\R^d):\nabla\cdot\u=0\},\\
	\mathbb{H}&:=\text{the closure of }\ \mathcal{V} \ \text{ in the Lebesgue space } \L^2(\mathbb{T}^d)=\mathrm{L}^2(\mathbb{T}^d;\R^d),\\
	\mathbb{V}&:=\text{the closure of }\ \mathcal{V} \ \text{ in the Sobolev space } \H^1(\mathbb{T}^d)=\mathrm{H}^1(\mathbb{T}^d;\R^d),\\
	\widetilde{\L}^{p}&:=\text{the closure of }\ \mathcal{V} \ \text{ in the Lebesgue space } \L^p(\mathbb{T}^d)=\mathrm{L}^p(\mathbb{T}^d;\R^d),
\end{align*}
for $p\in(2,\infty]$. The zero mean condition implies the  \emph{Poincar\'{e}-Wirtinger inequality}, \begin{align}\label{poin}
	\lambda_1\|\u\|_{\mathbb{H}}^2\leq\|\u\|^2_{\V},
\end{align} where $\lambda_1=\left(\frac{2\pi}{L}\right)^2$ (Lemma 5.40, \cite{Robinson1}). Then, we characterize the spaces $\H$, $\V$ and $\widetilde{\L}^p$   with the norms  $$\|\u\|_{\H}^2:=\int_{\mathbb{T}^d}|\u(x)|^2\d x,\quad \|\u\|_{\V}^2:=\int_{\mathbb{T}^d}|\nabla\u(x)|^2\d x , \quad \|\u\|_{\widetilde{\L}^p}^p=\int_{\mathbb{T}^d}|\u(x)|^p\d x,\ \text{ for }\ p\in(2,\infty),$$  and $\|\u\|_{\wi\L^{\infty}}=\esssup\limits_{x\in\mathbb{T}^d}|\u(x)|,$ respectively. 
Let us denote $(\cdot,\cdot)$  for the inner product in the Hilbert space $\H$ and $\langle \cdot,\cdot\rangle $ for the induced duality between the spaces $\V$  and its dual $\V'$ as well as $\widetilde{\L}^p$ and its dual $\widetilde{\L}^{p'}$, where $\frac{1}{p}+\frac{1}{p'}=1$. Note that $\H$ can be identified with its own dual $\H'$. The sum space $\V'+\widetilde{\L}^{p'}$ is well defined (see subsection 2.1, \cite{RFHK}). Furthermore, we infer
\begin{align*}
	(\V'+\widetilde{\L}^{p'})'=	\V\cap\widetilde{\L}^p \  \text{and} \ (\V\cap\widetilde{\L}^p)'=\V'+\widetilde{\L}^{p'},
\end{align*} 
where $\|\y\|_{\V\cap\wi\L^{p}}=\max\{\|\y\|_{\V},\|\y\|_{\wi\L^p}\},$ which is equivalent to the norms  $\|\y\|_{\V}+\|\y\|_{\widetilde{\L}^{p}}$  and $\sqrt{\|\y\|_{\V}^2+\|\y\|_{\widetilde{\L}^{p}}^2}$, and  
\begin{align*}
	\|\y\|_{\V'+\widetilde{\L}^{p'}}&=\inf\{\|\y_1\|_{\V'}+\|\y_2\|_{\wi\L^{p'}}:\y=\y_1+\y_2, \y_1\in\V' \ \text{and} \ \y_2\in\wi\L^{p'}\}\nonumber\\&=
	\sup\left\{\frac{|\langle\y_1+\y_2,\f\rangle|}{\|\f\|_{\V\cap\widetilde{\L}^p}}:\boldsymbol{0}\neq\f\in\V\cap\widetilde{\L}^p\right\}.
\end{align*}
Note that $\V\cap\widetilde{\L}^p$ and $\V'+\widetilde{\L}^{p'}$ are Banach spaces. Moreover, we have the continuous embedding $\V\cap\widetilde{\L}^p\hookrightarrow\V\hookrightarrow\H\hookrightarrow\V'\hookrightarrow\V'+\widetilde{\L}^{p'}$.  By Sobolev's embedding, we have  $\V\hookrightarrow\wi\L^p$ for all $p\in[2,\infty)$ in 2D and $p\in[2,6]$ in 3D.  

\iffalse 
We use the following interpolation inequality in the sequel. 
Assume $1\leq s_1\leq s\leq s_2\leq \infty$, $\theta\in(0,1)$ such that $\frac{1}{s}=\frac{\theta}{s_1}+\frac{1-\theta}{s_2}$ and $\u\in\L^{s_1}(\mathbb{T}^d)\cap\L^{s_2}(\mathbb{T}^d)$, then we have 
\begin{align}\label{211}
	\|\u\|_{\L^s(\mathbb{T}^d)}\leq\|\u\|_{\L^{s_1}(\mathbb{T}^d)}^{\theta}\|\u\|_{\L^{s_2}(\mathbb{T}^d)}^{1-\theta}. 
\end{align}
Note that Agmon's inequality yields for $\alpha_1<\frac{s}{2}<s_2$
\begin{align}\label{23}
	\|\u\|_{\L^{\infty}(\mathbb{T}^d)}\leq C\|\u\|_{\H^{\alpha_1}(\mathbb{T}^d)}^{\theta}\|\u\|_{\H^{\alpha_2}(\mathbb{T}^d)}^{1-\theta},
\end{align}
for $\u\in\H^{\alpha_2}(\mathbb{T}^d)$. 
\fi 
\subsection{Linear operator}
Let $\mathcal{P}: \mathring{\L}^2(\mathbb{T}^d) \to\H$ denote the Helmholtz-Hodge (or Leray) projection (section 2.1, \cite{RRS}).  We define the Stokes operator 
\begin{equation*}
	\A\u:=-\mathcal{P}\Delta\u,\;\u\in\D(\A):=\V\cap\mathring{\H}^{2}_p(\mathbb{T}^d).
\end{equation*}
Note that $\D(\A)$ can also be written as $\D(\A)=\big\{\u\in\mathring{\H}^{2}_p(\mathbb{T}^d):\nabla\cdot\u=0\big\}$.  It should be noted that $\mathcal{P}$ and $\Delta$ commutes in periodic domains (Lemma 2.9, \cite{RRS}). For the Fourier expansion $\u=\sum\limits_{\k\in\mathbb{Z}^d} e^{2\pi i \k\cdot\x /  L}\u_{\k},$ one obtains $$-\Delta\u=\left(\frac{2\pi}{L}\right)^2\sum\limits_{\k\in\mathbb{Z}^d} e^{2\pi i \k\cdot\x /  L}|\k|^2\u_{\k}.$$  The operator $\A$ is a non-negative self-adjoint operator in $\H$ with a compact resolvent and   \begin{align}\label{2.7a}\langle \A\u,\u\rangle =\|\u\|_{\V}^2,\ \textrm{ for all }\ \u\in\V, \ \text{ and }\ \|\A\u\|_{\V'}\leq \|\u\|_{\V}.\end{align}
Since $\A^{-1}$ is a compact self-adjoint operator in $\H$, we obtain a complete family of orthonormal eigenfunctions  $\{\boldsymbol{e}_k\}_{k=1}^{\infty}\subset\mathring{\C}_p^{\infty}(\mathbb{T}^d;\R^d)$ such that $\A \boldsymbol{e}_k=\lambda_k\boldsymbol{e}_k$, for $k=1,2,\ldots,$ and  $0<\lambda_1\leq \lambda_2\leq \ldots\to\infty$ are the eigenvalues of $\A$. Note that $\lambda_1=\left(\frac{2\pi}{L}\right)^2$ is the smallest eigenvalue of $\A$ appearing in the Poincar\'e-Wirtinger inequality \eqref{poin}.

In the sequel, we require the fractional powers of $\A$ also.  It is easy to observe that $\D(\A^{\frac{\alpha}{2}})=\big\{\u\in \mathring{\H}^{\alpha}_p(\mathbb{T}^d):\nabla\cdot\u=0\big\}$ and $\|\A^{\frac{\alpha}{2}}\u\|_{\H}=C\|\u\|_{\mathring{\H}^{\alpha}_p},$ for all $\u\in\D(\A^{\frac{\alpha}{2}})$, $\alpha\geq 0$ (cf. \cite{Robinson1}). For $\u\in \H$ and  $\alpha>0,$ one can define
$\A^\alpha \u=\sum\limits_{k=1}^\infty \lambda_k^\alpha \u_k \boldsymbol{e}_k,  \ \u\in\D(\A^\alpha), $ where $\D(\A^\alpha)=\left\{\u\in \H:\sum\limits_{k=1}^\infty \lambda_k^{2\alpha}|\u_k|^2<+\infty\right\}.$ 
Here  $\D(\A^\alpha)$ is equipped with the norm 
$
	\|\A^\alpha \u\|_{\H}=\left(\sum\limits_{k=1}^\infty \lambda_k^{2\alpha}|\u_k|^2\right)^{\frac{1}{2}}.
$
It can be easily seen that $\D(\A^0)=\H$, $\D(\A^\frac{1}{2})=\V$ and $\D(\A^{-\frac{1}{2}})=\V'$. We set $\V_\alpha= \D(\A^{\frac{\alpha}{2}})$ with $\|\u\|_{\V_{\alpha}} =\|\A^{\frac{\alpha}{2}} \u\|_{\H}.$

\subsection{Bilinear operator}
Let us define the \emph{trilinear form} $b(\cdot,\cdot,\cdot):\V\times\V\times\V\to\R$ by $$b(\u,\v,\w)=\int_{\mathbb{T}^d}(\u(x)\cdot\nabla)\v(x)\cdot\w(x)\d x=\sum\limits_{i,j=1}^d\int_{\mathbb{T}^d}\u_i(x)\frac{\partial \v_j(x)}{\partial x_i}\w_j(x)\d x.$$ If $\u, \v$ are such that the linear map $b(\u, \v, \cdot) $ is continuous on $\V$, the corresponding element of $\V'$ is denoted by $\B(\u, \v)$. We also denote  $\B(\u) = \B(\u, \u)=\mathcal{P}[(\u\cdot\nabla)\u]$.
An integration by parts yields  
\begin{equation}\label{b0}
	\left\{
	\begin{aligned}
		b(\u,\v,\v) &= 0,\ \text{ for all }\ \u,\v \in\V,\\
		b(\u,\v,\w) &=  -b(\u,\w,\v),\ \text{ for all }\ \u,\v,\w\in \V.
	\end{aligned}
	\right.\end{equation}
For $r\in[1,3]$,  using H\"older's inequality, we have 
$
\left|\langle \B(\u,\u),\v\rangle \right|=\left|b(\u,\v,\u)\right|\leq\|\u\|_{\wi\L^4}^2\|\v\|_{\V},
$ for all $\v\in\V$ so that $$\|\B(\u)\|_{\V'}\leq\|\u\|_{\wi\L^4}^2, \ \text{ for all }\ \u\in\wi\L^4, $$ and we conclude that $\B(\cdot):\V\cap\widetilde{\L}^{4}\to\V'+\widetilde{\L}^{\frac{4}{3}}$. Furthermore, we have 
\begin{align}\label{2p3}
	\|\B(\u)-\B(\v)\|_{\V'}&\leq \left(\|\u\|_{\widetilde{\L}^{4}}+\|\v\|_{\widetilde{\L}^{4}}\right)\|\u-\v\|_{\widetilde{\L}^{4}},
\end{align}
hence $\B(\cdot):\V\cap\widetilde{\L}^{4}\to\V'+\widetilde{\L}^{\frac{4}{3}}$ is a locally Lipschitz operator. 
An application of H\"older's inequality yields
\begin{align*}
	|b(\u,\v,\w)|=|b(\u,\w,\v)|\leq \|\u\|_{\widetilde{\L}^{r+1}}\|\v\|_{\widetilde{\L}^{\frac{2(r+1)}{r-1}}}\|\w\|_{\V},
\end{align*}
for all $\u\in\widetilde{\L}^{r+1}$, $\v\in\widetilde{\L}^{\frac{2(r+1)}{r-1}}$ and $\w\in\V$, so that we obtain  
\begin{align}\label{2p9}
	\|\B(\u,\v)\|_{\V'}\leq \|\u\|_{\widetilde{\L}^{r+1}}\|\v\|_{\widetilde{\L}^{\frac{2(r+1)}{r-1}}}\leq\|\u\|_{\wi\L^{r+1}}\|\v\|_{\wi\L^{r+1}}^{\frac{2}{r-1}}\|\v\|_{\H}^{\frac{r-3}{r-1}},
\end{align}
using the  interpolation inequality. Therefore, we deduce 
\begin{align}\label{212}
	\left|\langle \B(\u,\u),\v\rangle \right|=\left|b(\u,\v,\u)\right|\leq \|\u\|_{\widetilde{\L}^{r+1}}\|\u\|_{\widetilde{\L}^{\frac{2(r+1)}{r-1}}}\|\v\|_{\V}\leq\|\u\|_{\widetilde{\L}^{r+1}}^{\frac{r+1}{r-1}}\|\u\|_{\H}^{\frac{r-3}{r-1}}\|\v\|_{\V},
\end{align}
{for $r> 3$} and  all $\v\in\V$, which results to 
\begin{align}\label{2.9a}
	\|\B(\u)\|_{\V'}\leq\|\u\|_{\widetilde{\L}^{r+1}}^{\frac{r+1}{r-1}}\|\u\|_{\H}^{\frac{r-3}{r-1}}.
\end{align}
Using \eqref{2p9}, for $\u,\v\in\widetilde{\L}^{r+1}$, we also obtain 
\begin{align}\label{lip}
	\|\B(\u)-\B(\v)\|_{\V'}&\leq \left(\|\u\|_{\H}^{\frac{r-3}{r-1}}\|\u\|_{\widetilde{\L}^{r+1}}^{\frac{2}{r-1}}+\|\v\|_{\H}^{\frac{r-3}{r-1}}\|\v\|_{\widetilde{\L}^{r+1}}^{\frac{2}{r-1}}\right)\|\u-\v\|_{\widetilde{\L}^{r+1}},
\end{align}
{for $r> 3$}, by using the  interpolation inequality.  Therefore, the map $\B(\cdot):\V\cap\wi\L^{r+1}\to\V'+\wi\L^{\frac{r+1}{r}}$ is locally Lipschitz. 
\subsection{Nonlinear operator}
Let us now consider the nonlinear  operator $\mathcal{C}(\u):=\mathcal{P}(|\u|^{r-1}\u)$, for all $\u\in\wi\L^{r+1}$. It can be easily verified that $\langle\mathcal{C}(\u),\u\rangle =\|\u\|_{\widetilde{\L}^{r+1}}^{r+1}$. % Furthermore, for all $\u\in\wi\L^{r+1}$, the map is Gateaux differentiable with the Gateaux derivative 
%	\begin{align}\label{29}
	%	\mathcal{C}'(\u)\v&=\left\{\begin{array}{cl}\mathcal{P}(\v),&\text{ for }r=1,\\ \left\{\begin{array}{cc}\mathcal{P}(|\u|^{r-1}\v)+(r-1)\mathcal{P}\left(\frac{\u}{|\u|^{3-r}}(\u\cdot\v)\right),&\text{ if }\u\neq \mathbf{0},\\\mathbf{0},&\text{ if }\u=\mathbf{0},\end{array}\right.&\text{ for } 1<r<3,\\ \mathcal{P}(|\u|^{r-1}\v)+(r-1)\mathcal{P}(\u|\u|^{r-3}(\u\cdot\v)), &\text{ for }r\geq 3,\end{array}\right.
	%	\end{align} for $\v\in\wi\L^{r+1}$. 
For $0<\theta<1$, an application of Taylor's formula yields (\cite{MTM5,MTM4})
\begin{align}\label{213}
	|	\langle \mathcal{C}(\u)-\mathcal{C}(\v),\w\rangle|&%\leq \|(|\u|^{r-1}\u)-(|\v|^{r-1}\v)\|_{\widetilde{\L}^{\frac{r+1}{r}}}\|\w\|_{\widetilde{\L}^{r+1}}\nonumber\\&\leq \sup_{0<\theta<1}\|\mathcal{C}'(\theta\u+(1-\theta)\v)(\u-\v)|\|_{\widetilde{\L}^{{\frac{r+1}{r}}}}\|\w\|_{\widetilde{\L}^{r+1}}\nonumber\\&
	\leq r\left(\|\u\|_{\widetilde{\L}^{r+1}}+\|\v\|_{\widetilde{\L}^{r+1}}\right)^{r-1}\|\u-\v\|_{\widetilde{\L}^{r+1}}\|\w\|_{\widetilde{\L}^{r+1}},
\end{align}
for all $\u,\v,\w\in\widetilde{\L}^{r+1}$. 
Thus the operator $\mathcal{C}(\cdot):\widetilde{\L}^{r+1}\to\widetilde{\L}^{\frac{r+1}{r}}$ is locally Lipschitz. Furthermore, 	for any $r\in[1,\infty)$, we have (see \cite{MTM4})
\begin{align}\label{2.23}
	\langle\mathcal{C}(\u)-\mathcal{C}(\v),\u-\v\rangle&\geq \frac{1}{2}\||\u|^{\frac{r-1}{2}}(\u-\v)\|_{\H}^2+\frac{1}{2}\||\v|^{\frac{r-1}{2}}(\u-\v)\|_{\H}^2\nonumber\\&\geq \frac{1}{2^{r-1}}\|\u-\v\|_{\wi\L^{r+1}}^{r+1}\geq 0,
\end{align}
for $r\geq 1$ 	and all $\u,\v\in\wi\L^{r+1}$.  We also define the nonlinear  operator $\widetilde{\mathcal{C}}(\u):=\mathcal{P}(|\u|^{q-1}\u)$, for $q\in[1,r)$ and for all $\u\in\wi\L^{q+1}$. The operator $\widetilde{\mathcal{C}}(\cdot)$ also has the same properties as $\mathcal{C}(\cdot)$.

	\subsection{Solvability results} 
Let $(\Omega,\mathscr{F},\mathbb{P})$ be a probability space equipped with a filtration $\{\mathscr{F}_t\}_{t\geq 0},$ satisfying the usual conditions and $\W = \{(\W_1(t),\ldots,\W_m (t))\}_{t\geq 0}$  is an $m$-dimensional standard Brownian motion defined on it.	The system \eqref{1.5} can be reformulated as follows:
	\begin{equation}\label{2p11}
			\left\{
		\begin{aligned}
			\d\u(t)&+\left[\mu\A\u(t)+\B(\u(t))+\alpha\widetilde{\mathcal{C}}(\u(t))+\beta\mathcal{C}(\u(t))\right]\d t\\&=\F(\u(t))\d t+\sum\limits_{i=1}^m\sigma^i(\u(t))\d\W^i(t),\\
			\u(0)&=\h,
		\end{aligned}
		\right.
	\end{equation}
where $\mathcal{P}\F$ and $\mathcal{P}\sigma^i$ are denoted by $\F$ and $\sigma^i$, respectively for simplicity of notations. Let $\F,\sigma_i,i=1,\ldots,m$ be measurable mappings from $\H$ into $\H$ and satisfy the following assumption: 
\begin{hypothesis}\label{hyp1}
	(H.1) The mappings $\F,\sigma_i:\H\to\H$ are globally Lipschitz maps, that is, there exists a positive constant $L_1<\infty$	such that
	\begin{align}
		\|\F(\u_1)-\F(\u_2)\|_{\H}^2+\sum\limits_{i=1}^m\|\sigma^i(\u_1)-\sigma^i(\u_2)\|_{\H}^2\leq L_1\|\u_1-\u_2\|_{\H}^2,\ \text{ for all }\ \u_1,\u_2\in\H. 
	\end{align}
\end{hypothesis}
From Hypothesis \ref{hyp1}, it is clear that 
\begin{align}
	\|\F(\u)\|_{\H}^2+\sum\limits_{i=1}^m\|\sigma^i(\u)\|_{\H}^2\leq K\left(1+\|\u\|_{\H}^2\right),
\end{align}
where $K=2\max\left\{(n+1)L,\sum\limits_{i=1}^m\|\sigma^i(\boldsymbol{0})\|_{\H}^2+\|\F(\boldsymbol{0})\|_{\H}^2\right\}$.  Let us now provide the definition of strong solution in the probabilistic sense (weak solution in the analytic sense) to the system \eqref{2p11}. 	
\begin{definition}
	A continuous $\H$-continuous $\mathscr{F}_t$ adapted process $\u=\{\u(t)\}_{t\geq 0}$ is said to be a \emph{strong solution} to the system \eqref{2p11} if for any $T > 0$, $\u\in\mathrm{L}^2(\Omega;\mathrm{L}^{\infty}(0,T;\H)\cap\mathrm{L}^2(0,T;\V))\cap\mathrm{L}^{r+1}(\Omega;\mathrm{L}^{r+1}(0,T;\wi\L^{r+1}))$ and for any $t\geq 0$, the
	following equation holds $\mathbb{P}$-a.s.:
	\begin{align}\label{2p13}
		(\u(t),\v)&=(\boldsymbol{h},\v)-\mu\int_0^t\langle\A\u(s),\v\rangle\d s-\int_0^t\langle\B(\u(s)),\v\rangle\d s-\alpha\int_0^t\langle\widetilde{\mathcal{C}}(\u(s)),\v\rangle\d s\nonumber\\&\quad-\beta\int_0^t\langle\mathcal{C}(\u(s)),\v\rangle\d s+\int_0^t(\F(\u(s)),\v)\d s+\sum\limits_{i=1}^m\int_0^t(\boldsymbol{\sigma}^{i}(\u(s)),\v)\d\W^i(s),
	\end{align}
for all $\v\in\V\cap\wi\L^{r+1}$. A strong solution $\u(\cdot)$ to \eqref{2p11} is called a \emph{pathwise  unique strong solution} if $\widetilde{\u}(\cdot)$ is an another strong solution, then $$\mathbb{P}\Big\{\omega\in\Omega:  \u(t)=\widetilde{\u}(t),\  \text{ for all }\ t\in[0,T]\Big\}=1.$$ 
\end{definition}
Under Hypothesis \ref{hyp1} and $\h\in\H$, from \cite[Theorem 3.7]{MTM4},  it is known that the system \eqref{2p11}  admits a unique strong solution for $d=2$, $r\in[1,\infty)$ and $d=3$, $r\in[3,\infty)$. Moreover, the following It\^o's formula holds true:
\begin{align}
		&\|\u(t)\|_{\H}^2+2\mu\int_0^t\|\u(s)\|_{\V}^2\d s+2\beta\int_0^t\|\u(s)\|_{\wi\L^{r+1}}^{r+1}\d s\nonumber\\&=\|\h\|^2-2\alpha\int_0^t\|\u(s)\|_{\wi\L^{q+1}}^{q+1}\d s+2\int_0^t(\F(\u(s)),\u(s))\d s+\sum\limits_{i=1}^m\int_0^t\|\sigma^i(\u(s))\|_{\H}^2\d s\nonumber\\&\quad+2\sum\limits_{i=1}^m\int_0^t(\boldsymbol{\sigma}^{i}(\u(s)),\u(s))\d\W^i(s),\ \mathbb{P}\text{-a.s.},
\end{align}
 for all $t\geq 0$. These solutions are weak in the analytical sense (derivatives exists only in the sense of distributions) and strong in the stochastic sense (the underlying probability space is a priori given). 
	
	\section{Approximations of stochastic CBFeD equations by pure jump type stochastic CBFeD equations}\label{Sec3}\setcounter{equation}{0}
Let  $\lambda^i(\d z)$,  $i=1,\ldots,m$ denote $\sigma$-finite measures on the measurable space $(\R_0,\mathcal{B}(\R_0))$, where $\R_0 :=\R\backslash\{0\}$. Let $\pi^i,\ i=1,\ldots,m$ be mutually independent $\mathscr{F}_t$-adapted Poisson random measures on $[0,T]\times\R_0$ with intensity measure $\d t\times\lambda^i(\d z)$. For $U\in\mathcal{B}(\R_0)$ with $\lambda^i(U)<\infty$, we define $\widetilde{\pi}^i((0,t]\times U):=\widetilde{\pi}^i((0,t]\times U)-t\lambda^i(U),\ t>0,$ for the corresponding compensated Poisson random measures on $[0, T ] \times\Omega\times\R_0$. 	 For more details on Poisson random measures, we refer the interested readers to \cite{Ap,IW}, etc. 

For $\e>0$, let $\boldsymbol{\sigma}^{i,\e}:\H\times\mathbb{R}_0\to\mathbb{R}$, $i=1,\ldots,m$ be given measurable maps. Let us consider the following stochastic convective Brinkman-Forchheimer-extended
	Darcy (CBFeD) equations perturbed by pure jump noise: 
	\begin{align}\label{3p1}
		\u_{\e}(t)&=\boldsymbol{h}-\mu\int_0^t\A\u_{\e}(s)\d s-\int_0^t\B(\u_{\e}(s))\d s-\alpha\int_0^t\widetilde{\mathcal{C}}(\u_{\e}(s))\d s-\beta\int_0^t\mathcal{C}(\u_{\e}(s))\d s\nonumber\\&\quad+\int_0^t\F(\u_{\e}(s))\d s+\sum\limits_{i=1}^m\int_0^t\int_{\R_0}\boldsymbol{\sigma}^{i,\e}(\u_{\e}(s-),z)\wi{\pi}^{i}(\d s,\d z),
	\end{align}
in $\V'+\wi\L^{\frac{r+1}{r}}$. In order to obtain the global solvability results of the system \eqref{3p1}, we impose the following conditions on $\boldsymbol{\sigma}^{i,\e}$: 
\begin{hypothesis}\label{hyp2}
	(H.2) There exist constants $K_1,K_2,L_2>0$ and $\e_0>0$ such that 
		\begin{align}
			&\|\F(\u)\|_{\H}^2+\sup_{0<\e\leq \e_0}\sum\limits_{i=1}^m\int_{\R_0}\|\boldsymbol{\sigma}^{i,\e}(\u,z)\|_{\H}^2\lambda^{i}(\d z)\leq K_1(1+\|\u\|_{\H}^2), \label{32}\\
			&\sup_{0<\e\leq \e_0}\sum\limits_{i=1}^m\int_{\R_0}\|\boldsymbol{\sigma}^{i,\e}(\u,z)\|_{\H}^{2p}\lambda^{i}(\d z)\leq K_2(1+\|\u\|_{\H}^{2p}), \label{33}\\
			&\|\F(\u_1)-\F(\u_2)\|_{\H}^2+\sup_{0<\e\leq \e_0}\sum\limits_{i=1}^m\int_{\R_0}\|\boldsymbol{\sigma}^{i,\e}(\u_1,z)-\boldsymbol{\sigma}^{i,\e}(\u_2,z)\|_{\H}^2\lambda^{i}(\d z)\nonumber\\&\leq L_2\|\u_1-\u_2\|_{\H}^2,\label{34}
		\end{align}
	where $p=\frac{1+\e}{2}\max\{3,r+1\}$. 
\end{hypothesis}
Let us	denote  $\D([0,T];\H)$ for the space of all c\`adl\`ag paths from $[0,T]$ into $\H$  equipped with the	Skorokhod topology.
\begin{definition}\label{def3.2}
	An $\H$-valued $\{\mathscr{F}_t\}$-adapted process $\u_{\e}=\{\u_{\e}(t)\}_{t\geq 0}$ is said to be a  strong solution to \eqref{3p1} if 
	\begin{itemize}
		\item [(i)] for any $T>0$, $\u_{\e}\in\mathrm{L}^2(\Omega;\mathrm{L}^{\infty}(0,T;\H)\cap\mathrm{L}^2(0,T;\H))\cap\mathrm{L}^{r+1}(\Omega;\mathrm{L}^{r+1}(0,T;\wi\L^{r+1}))$ having a  modification with paths in $\D([0,T];\H)$, $\mathbb{P}$-a.s.,
		\item [(ii)] for every $t\geq 0$, 
		\begin{align}\label{3p5}
			(\u_{\e}(t),\v)&=(\boldsymbol{h},\v)-\mu\int_0^t\langle\A\u_{\e}(s),\v\rangle\d s-\int_0^t\langle \B(\u_{\e}(s)),\v\rangle\d s-\alpha\int_0^t\langle\widetilde{\mathcal{C}}(\u_{\e}(s)),\v\rangle\d s\nonumber\\&\quad-\beta\int_0^t\langle\mathcal{C}(\u_{\e}(s)),\v\rangle\d s+\int_0^t(\F(\u_{\e}(s)),\v)\d s\nonumber\\&\quad+\sum\limits_{i=1}^m\int_0^t\int_{\R_0}(\boldsymbol{\sigma}^{i,\e}(\u_{\e}(s-),z),\v)\wi{\pi}^{i}(\d s,\d z),\ \mathbb{P}\text{-a.s.},
		\end{align}
	for $\v\in\V\cap\wi\L^{r+1}$, 
	\end{itemize}
A strong solution $\u_{\e}(\cdot)$ to the system (\ref{3p1}) is called a 	\emph{pathwise  unique strong solution} if	$\widetilde{\u}_{\e}(\cdot)$ is an another strong	solution, then $$\mathbb{P}\Big\{\omega\in\Omega:\u_{\e}(t)=\widetilde{\u}_{\e}(t),\ \text{ for all }\ t\in[0,T]\Big\}=1.$$ 
\end{definition}
Under Hypothesis  \ref{hyp2} (H.2) and $\h\in\H$, it is well-known that for $0<\e\leq\e_0$, the system \eqref{3p1} admits a unique strong solution solution in the sense of Definition \ref{def3.2} (see Theorem 3.6, \cite{MTM6}). Moreover, the following It\^o formula holds true for all $t\geq 0$, $\mathbb{P}$\text{-a.s.}
\begin{align}\label{3p6}
	&	\|\u_{\e}(t)\|_{\H}^2+2\mu\int_0^t\|\u_{\e}(s)\|_{\V}^2\d s+2\beta\int_0^t\|\u_{\e}(s)\|_{\wi\L^{r+1}}^{r+1}\d s\nonumber\\&=\|\h\|^2-2\alpha\int_0^t\|\u_{\e}(s)\|_{\wi\L^{q+1}}^{q+1}\d s+2\int_0^t(\F(\u_{\e}(s)),\u_{\e}(s))\d s\nonumber\\&\quad+\sum\limits_{i=1}^m\int_0^t\int_{\R_0}\|\boldsymbol{\sigma}^{i,\e}(\u_{\e}(s),z)\|_{\H}^2\pi^i(\d s,\d z)+2\sum\limits_{i=1}^m\int_0^t\int_{\R_0}(\boldsymbol{\sigma}^{i,\e}(\u_{\e}(s-),z),\u_{\e}(s-))\wi{\pi}^{i}(\d s,\d z).
\end{align}

In order to achieve the goal of this work, we consider the following conditions:
\begin{hypothesis}\label{hyp3}
(H.3) \begin{enumerate}
		\item [(i)] For each $i\in\{1,\ldots,m\}$, for all $M>0$, 
		\begin{align}\label{3p7}
			\sup_{\|\u\|_{\H}\leq M}\sup_{z\in\mathbb{R}_0}\|\sigma^{i,\e}(\u,z)\|_{\H}\xrightarrow{\e\to 0}0 . 
		\end{align}
	\item [(ii)] For each $i\in\{1,\ldots,m\}$ and each $k,j\in\mathbb{N}$, $\u\in\H$, 
	\begin{align}
		\int_{\mathbb{R}_0}(\sigma^{i,\e}(\u,z),\boldsymbol{e}_k)(\sigma^{i,\e}(\u,z),\boldsymbol{e}_j)\lambda^{i}(\d z)\xrightarrow{\e\to 0}(\sigma^i(\u),\boldsymbol{e}_k)(\sigma^i(\u),\boldsymbol{e}_j). 
	\end{align}
	\end{enumerate}
(H.4) For each $i\in\{1,\ldots,m\}$ and every $\u\in\H$, 
\begin{align}
	\int_{\mathbb{R}_0}\|\sigma^{i,\e}(\u,z)\|_{\H}^2\lambda^i(\d z)\xrightarrow{\e\to 0}\|\sigma^{i}(\u)\|_{\H}^2. 
\end{align}
\end{hypothesis}
Motivation for considering  Hypothesis \ref{hyp3} (H.3)-(H.4) is discussed in Remark 3.2, \cite{SSTZ}.  Condition (i) of (H.3) is introduced in response to the intuition of approximating Brownian motion through pure jump noise by requiring that the jump heights of all jumps should converge to zero. Applying Ito's formula to $\|\u_{\e}(\cdot)\|_{\H}^2$, we introduce (H.4) in order to provide an approximation of the $\H$-norm of the solution of \eqref{2p13} in some sense. Condition (ii) of (H.3) is introduced to justify the limit of the solutions of \eqref{3p1} is a probabilistic weak solution of \eqref{2p13} through the associated martingale problem. The infinite volume of jump measures $\lambda^i,\ i=1,\ldots,m$ is necessary, as (H.3)'s condition (i) and (H.4) contradict each other by the dominated convergence theorem, if they have finite volume.

We need the following assumption on $\sigma^{i,\e}(\cdot,\cdot)$ to obtain the regularity results of the system \eqref{3p1}. 
\begin{hypothesis}\label{hyp4}
	\begin{enumerate}
		\item [(H.5)] The map $\sigma^{i,\e}$ takes the space $\V$ into itself and there exist constants $C > 0$ and $\e_0 > 0$	such that
	\begin{align}
\sup_{0<\e\leq \e_0}\sum\limits_{i=1}^m\int_{\mathbb{R}_0}\|\sigma^{i,\e}(\u,z)\|_{\V}^2\lambda^i(\d z)\leq C\left(1+\|\u\|_{\V}^2\right). 
	\end{align}
	\end{enumerate}
\end{hypothesis}

\subsection{Energy estimates and tightness}
In the rest of the paper, we take $m = 1$ for simplicity  and omit the superscript $i$ of $\sigma^{i},\wi\pi^{i},\lambda^i$. The case of $m > 1$ does not cause any extra difficulties.

\begin{lemma}\label{lem3.6}
	Assume that  Hypothesis \ref{hyp2} (H.2) holds and $\h\in\H$. Let $\u_{\e}(\cdot)$ be the  unique strong solution of \eqref{3p1}. Then we have 
	\begin{align}\label{3p12}
		\sup_{0<\e\leq\e_0}\left\{\E\left[\sup_{0\leq t\leq T}\|\u_{\e}(t)\|_{\H}^{2p}\right]+\E\left[\left(\int_0^T\|\u_{\e}(t)\|_{\V}^2\d t\right)^p\right]+\E\left[\left(\int_0^T\|\u_{\e}(t)\|_{\wi\L^{r+1}}^{r+1}\d t\right)^p\right]\right\}<\infty,
	\end{align}
for  $p=\frac{1+\e}{2}\max\{3,r+1\}$. 
\end{lemma}
\begin{proof} By  an application of the infinite dimensional It\^o formula \eqref{3p6}, we have 
		\begin{align}\label{3p13}
		&	\|\u_{\e}(t)\|_{\H}^2+2\mu\int_0^t\|\u_{\e}(s)\|_{\V}^2\d s+2\beta\int_0^t\|\u_{\e}(s)\|_{\wi\L^{r+1}}^{r+1}\d s\nonumber\\&=\|\h\|^2-2\alpha\int_0^t\|\u_{\e}(s)\|_{\wi\L^{q+1}}^{q+1}\d s+2\int_0^t(\F(\u_{\e}(s)),\u_{\e}(s))\d s+\int_0^t\int_{\R_0}\|\boldsymbol{\sigma}^{\e}(\u_{\e}(s),z)\|_{\H}^2\pi(\d s,\d z)\nonumber\\&\quad+2\int_0^t\int_{\R_0}(\boldsymbol{\sigma}^{\e}(\u_{\e}(s-),z),\u_{\e}(s-))\wi{\pi}(\d s,\d z),\ \mathbb{P}\text{-a.s.}, 
	\end{align}
for all $t\in[0,T]$. Taking supremum over time $t\in[0,T]$ and then taking expectation in \eqref{3p13}, we get 
\begin{align}\label{3p14}
&\E\left[\sup_{t\in[0,T]}	\|\u_{\e}(t)\|_{\H}^2\right]+2\mu\E\left[\int_0^T\|\u_{\e}(t)\|_{\V}^2\d t\right]+\beta\E\left[\int_0^T\|\u_{\e}(t)\|_{\wi\L^{r+1}}^{r+1}\d t\right]\nonumber\\&\leq\|\h\|_{\H}^2+C_{\alpha,\beta} +\E\left[\int_0^T\|\F(\u_{\e}(t))\|_{\H}^2\d t\right]+\E\left[\int_0^T\|\u_{\e}(t)\|_{\H}^2\d t\right]\nonumber\\&\quad+\E\left[\int_0^T\int_{\R_0}\|\boldsymbol{\sigma}^{\e}(\u_{\e}(t),z)\|_{\H}^2\lambda(\d z)\d t\right]\nonumber\\&\quad+2\E\left[\sup_{t\in[0,T]}\left|\int_0^t\int_{\R_0}(\boldsymbol{\sigma}^{\e}(\u_{\e}(s-),z),\u_{\e}(s-))\wi{\pi}(\d s,\d z)\right|\right],
\end{align}
where $C_{\alpha,\beta}=(2\alpha)^{\frac{r+1}{r-q}}\left(\frac{r-q}{r+1}\right)\left(\frac{q+1}{\beta(r+1)}\right)^{\frac{q+1}{r-q}}|\mathbb{T}^d|$, for $1\leq q<r$ and $|\mathbb{T}^d|$ is the volume of $\mathbb{T}^d$. Applying Burkholder-Davis-Gundy's and Young's inequalities, we estimate the final term from the right hand side of the inequality \eqref{3p14} as 
\begin{align}\label{3p15}
	I&\leq 2\sqrt{3} \E\left[\left(\int_0^T\int_{\R_0}\|\boldsymbol{\sigma}^{\e}(\u_{\e}(t),z)\|_{\H}^2\|\u_{\e}(t)\|_{\H}^2\pi(\d t,\d z)\right)^{\frac{1}{2}}\right]\nonumber\\&\leq 2\sqrt{3}\E\left[\sup_{t\in[0,T]}\|\u_{\e}(t)\|_{\H}\left(\int_0^T\int_{\R_0}\|\boldsymbol{\sigma}^{\e}(\u_{\e}(t),z)\|_{\H}^2\pi(\d t,\d z)\right)^{\frac{1}{2}}\right]\nonumber\\&\leq\frac{1}{2}\E\left[\sup_{t\in[0,T]}\|\u_{\e}(t)\|_{\H}^2\right]+6\E\left[\int_0^T\int_{\R_0}\|\boldsymbol{\sigma}^{\e}(\u_{\e}(t),z)\|_{\H}^2\lambda(\d z)\d t\right]. 
\end{align}
Using \eqref{3p15} in \eqref{3p14} and then applying Hypothesis \ref{hyp2},  we deduce 
\begin{align}\label{3p16}
	&\frac{1}{2}\E\left[\sup_{t\in[0,T]}	\|\u_{\e}(t)\|_{\H}^2\right]+2\mu\E\left[\int_0^T\|\u_{\e}(t)\|_{\V}^2\d t\right]+\beta\E\left[\int_0^T\|\u_{\e}(t)\|_{\wi\L^{r+1}}^{r+1}\d t\right]\nonumber\\&\leq\|\h\|_{\H}^2+C_{\alpha,\beta} +C\E\left[\int_0^T\left(\|\F(\u_{\e}(t))\|_{\H}^2+\int_{\R_0}\|\boldsymbol{\sigma}^{\e}(\u_{\e}(t),z)\|_{\H}^2\lambda(\d z)\right)\d t\right]\nonumber\\&\quad+\E\left[\int_0^T\|\u_{\e}(t)\|_{\H}^2\d t\right]\nonumber\\&\leq\|\h\|_{\H}^2+C_{\alpha,\beta}+(CK_1+1)\E\left[\int_0^T\left(1+\|\u_{\e}(t)\|_{\H}^2\right)\d t\right]. 
\end{align}
An application of Gronwall's inequality in \eqref{3p16} yields 
\begin{align}\label{3p17}
	&\E\left[\sup_{t\in[0,T]}	\|\u_{\e}(t)\|_{\H}^2\right]+4\mu\E\left[\int_0^T\|\u_{\e}(t)\|_{\V}^2\d t\right]+2\beta\E\left[\int_0^T\|\u_{\e}(t)\|_{\wi\L^{r+1}}^{r+1}\d t\right]\nonumber\\&\leq \left(\|\h\|_{\H}^2+(C_{\alpha,\beta}+CK_1+1)T\right)e^{(K_1+1)T}\leq C(\alpha,\beta,K_1,T)(1+\|\h\|_{\H}^2). 
\end{align}
Since each term on the left hand side of \eqref{3p13} is non-negative, one can easily deduce that 
\begin{align}\label{3p18}
		&	\E\left[\sup_{t\in[0,T]}\|\u_{\e}(t)\|_{\H}^{2p}+\left(\int_0^T\|\u_{\e}(t)\|_{\V}^2\d t\right)^{p}+\left(\int_0^T\|\u_{\e}(t)\|_{\wi\L^{r+1}}^{r+1}\d t\right)^{p}\right]\nonumber\\&\leq C\|\h\|_{\H}^{2p}+CC_{\alpha,\beta}^{2p}+C\E\left[\left(\int_0^T\|\F(\u_{\e}(t))\|_{\H}^2\d t\right)^{\frac{p}{2}}\left(\int_0^T\|\u_{\e}(t)\|_{\H}^2\d t\right)^{\frac{p}{2}}\right]\nonumber\\&\quad+C\E\left[\left(\int_0^T\int_{\R_0}\|\boldsymbol{\sigma}^{\e}(\u_{\e}(t),z)\|_{\H}^2\pi(\d z,\d t)\right)^p\right]\nonumber\\&\quad+C\E\left[\sup_{t\in[0,T]}\left|\int_0^t\int_{\R_0}(\boldsymbol{\sigma}^{\e}(\u_{\e}(s-),z),\u_{\e}(s-))\wi{\pi}(\d s,\d z)\right|^p\right]. 
\end{align}
Using Remark 2.5, \cite{JZZB} and Hypothesis \ref{hyp2}, we estimate the penultimate term from the right hand side of the inequality \eqref{3p18} as 
\begin{align}\label{3p19}
&	C\E\left[\left(\int_0^T\int_{\R_0}\|\boldsymbol{\sigma}^{\e}(\u_{\e}(t),z)\|_{\H}^2\pi(\d z,\d t)\right)^p\right]\nonumber\\&\leq C\E\left[\int_0^T\int_{\R_0}\|\boldsymbol{\sigma}^{\e}(\u_{\e}(t),z)\|_{\H}^{2p}
\lambda(\d z)\d t\right]+C\E\left[\left(\int_0^T\int_{\R_0}\|\boldsymbol{\sigma}^{\e}(\u_{\e}(t),z)\|_{\H}^2\lambda(\d z)\d t\right)^p\right]\nonumber\\&\leq CK_2\E\left[\int_0^T\left(1+\|\u_{\e}(t)\|_{\H}^{2p}\right)\d t\right]+CK_1^p\E\left[\left(\int_0^T\left(1+\|\u_{\e}(t)\|_{\H}^2\right)\d t\right)^p\right]\nonumber\\&\leq C(K_2+K_1^p)\left(T^p+(1+T^{p-1})\E\left[\int_0^T\|\u_{\e}(t)\|_{\H}^{2p}\d t\right]\right). 
	\end{align}
Making use of Burkholder-Davis-Gundy's and Young's inequalities, and \eqref{3p19}, we estimate the final term from the right hand side of the inequality \eqref{3p18} as 
\begin{align}\label{3p20}
	&C\E\left[\sup_{t\in[0,T]}\left|\int_0^t\int_{\R_0}(\boldsymbol{\sigma}^{\e}(\u_{\e}(s-),z),\u_{\e}(s-))\wi{\pi}(\d s,\d z)\right|^p\right]\nonumber\\&\leq C\E\left[\sup_{t\in[0,T]}\|\u_{\e}(t)\|_{\H}^{p}\left(\int_0^T\int_{\R_0}\|\boldsymbol{\sigma}^{\e}(\u_{\e}(t),z)\|_{\H}^2\pi(\d t,\d z)\right)^{\frac{p}{2}}\right]\nonumber\\&\leq\frac{1}{2}\E\left[\sup_{t\in[0,T]}\|\u_{\e}(t)\|_{\H}^{2p}\right]+\E\left[\left(\int_0^T\int_{\R_0}\|\boldsymbol{\sigma}^{\e}(\u_{\e}(t),z)\|_{\H}^2\pi(\d t,\d z)\right)^p\right]\nonumber\\&\leq\frac{1}{2}\E\left[\sup_{t\in[0,T]}\|\u_{\e}(t)\|_{\H}^{2p}\right]+C(K_2+K_1^p)\left(T^{p}+(1+T^{p-1})\E\left[\int_0^T\|\u_{\e}(t)\|_{\H}^{2p}\d t\right]\right). 
\end{align}
Moreover, using Hypothesis \ref{hyp2}, we have 
\begin{align}\label{3p21}
&	C\E\left[\left(\int_0^T\|\F(\u_{\e}(t))\|_{\H}^2\d t\right)^{\frac{p}{2}}\left(\int_0^T\|\u_{\e}(t)\|_{\H}^2\d t\right)^{\frac{p}{2}}\right]\nonumber\\&\leq C\E\left[\left(\int_0^T\|\F(\u_{\e}(t))\|_{\H}^2\d t\right)^{p}\right]+C\E\left[\left(\int_0^T\|\u_{\e}(t)\|_{\H}^2\d t\right)^{p}\right]\nonumber\\&\leq CK_1^pT^p+C(K_1+1)T^{p-1}\E\left[\int_0^T\|\u(t)\|_{\H}^{2p}\d t\right]. 
\end{align}
Combining \eqref{3p19}-\eqref{3p21} and substituting it in \eqref{3p18}, we get 
\begin{align}
	&	\E\left[\sup_{t\in[0,T]}\|\u_{\e}(t)\|_{\H}^{2p}+\left(\int_0^T\|\u_{\e}(t)\|_{\V}^2\d t\right)^{p}+\left(\int_0^T\|\u_{\e}(t)\|_{\wi\L^{r+1}}^{r+1}\d t\right)^{p}\right]\nonumber\\&\leq C\left[\|\h\|_{\H}^{2p}+C_{\alpha,\beta}^{2p}+(K_2+K_1^p)\left(T^p+(1+T^{p-1})\E\left[\int_0^T\|\u_{\e}(t)\|_{\H}^{2p}\d t\right]\right)\right]. 
\end{align}
An application of Gronwall's inequality yields 
\begin{align}
	&	\E\left[\sup_{t\in[0,T]}\|\u_{\e}(t)\|_{\H}^{2p}+\left(\int_0^T\|\u_{\e}(t)\|_{\V}^2\d t\right)^p+\left(\int_0^T\|\u_{\e}(t)\|_{\wi\L^{r+1}}^{r+1}\d t\right)^p\right]\nonumber\\&\leq C\left\{\|\h\|_{\H}^{2p}+C_{\alpha,\beta}^{2p}+(K_2+K_1^p)T\right\}e^{\left\{CT(K_2+K_1^p)(1+T)\right\}}\nonumber\\&\leq C\left(1+\|\h\|_{\H}^{2p}\right), 
\end{align}
and \eqref{3p12} follows. 
	\end{proof}

Using similar methods, one can show the following result for the solutions of the system \eqref{2p11}:
\begin{lemma}\label{lem3.7}
	Under Hypothesis \ref{hyp1} (H.1) and $\h\in\H$,  the following norm estimate holds for the solution $\u(\cdot)$ of the system \eqref{2p11}:
	\begin{align}
		\left\{\E\left[\sup_{0\leq t\leq T}\|\u(t)\|_{\H}^{2p}\right]+\E\left[\left(\int_0^T\|\u(t)\|_{\V}^2\d t\right)^p\right]+\E\left[\left(\int_0^T\|\u(t)\|_{\wi\L^{r+1}}^{r+1}\d t\right)^p\right]\right\}<\infty,
	\end{align}
	for any $p\geq 2$. 
\end{lemma}

\begin{lemma}\label{lem3.8}
Under Hypotheses \ref{hyp2} and \ref{hyp4}, for $\h\in\V$, we have 
	\begin{align}\label{323}
		&\sup_{0<\e\leq\e_0}\left\{\E\left[\sup_{0\leq t\leq T}\|\u_{\e}(t)\|_{\V}^2\right]+\E\left[\int_0^T\|\A\u_{\e}(t)\|_{\H}^2\d t\right]+\E\left[\int_0^T\|\u_{\e}(t)\|_{\wi\L^{p(r+1)}}^{r+1}\d t\right]\right\}\nonumber\\&\leq C(1+\|\h\|_{\V}^2)<\infty,
	\end{align}
where $p\in[2,\infty)$ for $d=2$ and $p=3$ for $d=3$.
\end{lemma}
\begin{proof}
	Through Galerkin approximations, one can  show that for $0<\e\leq\e_0$, the strong  solution $\u_{\e}$ of \eqref{3p1} has the regularity  $$\u_{\e}\in\mathrm{L}^2(\Omega;\mathrm{L}^{\infty}(0,T;\V)\cap\mathrm{L}^{2}(0,T;\D(\A)))\cap\mathrm{L}^{r+1}(\Omega;\mathrm{L}^{r+1}(0,T;\wi\L^{p(r+1)}))$$ ($p\in[2,\infty)$ for $d=2$ and $p=3$ for $d=3$) having a modification with paths in $\D([0,T];\V)$, $\mathbb{P}$-a.s. (cf. Theorem 3.11, \cite{MTM6}). Therefore, we derive the energy estimate \eqref{323} only. Applying It\^o's formula to the process $\|\u_{\e}(\cdot)\|_{\V}^2$, we find 
	\begin{align}\label{3p24}
	&	\|\u_{\e}(t)\|_{\V}^2+2\mu\int_0^t\|\A\u_{\e}(s)\|_{\H}^2\d s\nonumber\\&=\|\h\|_{\V}^2-2\int_0^t(\B(\u_{\e}(s)),\A\u_{\e}(s))\d s-2\alpha\int_0^t(\widetilde{\mathcal{C}}(\u_{\e}(s)),\A\u_{\e}(s))\d s\nonumber\\&\quad-2\beta\int_0^t(\mathcal{C}(\u_{\e}(s)),\A\u_{\e}(s))\d s+2\int_0^t(\F(\u_{\e}(s)),\A\u_{\e}(s))\d s\nonumber\\&\quad+\int_0^t\int_{\R_0}\|\boldsymbol{\sigma}^{\e}(\u_{\e}(s),z)\|_{\V}^2\pi(\d s,\d z)+2\int_0^t\int_{\R_0}(\nabla\boldsymbol{\sigma}^{\e}(\u_{\e}(s-),z),\nabla\u_{\e}(s-))\wi{\pi}(\d s,\d z),\ \mathbb{P}\text{-a.s.}, 
	\end{align}
Note that 
\begin{align}\label{3p25}
	&\int_{\mathbb{T}^d}(-\Delta\u_{\e}(x))\cdot|\u_{\e}(x)|^{r-1}\u_{\e}(x)\d x\nonumber\\&=\int_{\mathbb{T}^d}|\nabla\u_{\e}(x)|^2|\u_{\e}(x)|^{r-1}\d x+4\left[\frac{r-1}{(r+1)^2}\right]\int_{\mathbb{T}^d}|\nabla|\u_{\e}(x)|^{\frac{r+1}{2}}|^2\d x\nonumber\\&=\int_{\mathbb{T}^d}|\nabla\u_{\e}(x)|^2|\u_{\e}(x)|^{r-1}\d x+\frac{r-1}{4}\int_{\mathbb{T}^d}|\u_{\e}(x)|^{r-3}|\nabla|\u_{\e}(x)|^2|^2\d x.
\end{align}
On a torus, the operators $\mathcal{P}$ and $-\Delta$ commutes, and we have the following result (see Lemma 2.1, \cite{KWH}): 
\begin{align}\label{370}
	0&\leq\int_{\mathbb{T}^d}|\nabla\u_{\e}(x)|^2|\u_{\e}(x)|^{r-1}\d x\leq\int_{\mathbb{T}^d}|\u_{\e}(x)|^{r-1}\u_{\e}(x)\cdot\A\u_{\e}(x)\d x\nonumber\\&\leq r\int_{\mathbb{T}^d}|\nabla\u_{\e}(x)|^2|\u_{\e}(x)|^{r-1}\d x.
\end{align} 
By using \eqref{370}, the term $2\alpha\int_0^t(\widetilde{\mathcal{C}}(\u_{\e}(s)),\A\u_{\e}(s))\d s$ can be estimated as 
\begin{align}
	2\alpha\left|\int_0^t(\widetilde{\mathcal{C}}(\u_{\e}(s)),\A\u_{\e}(s))\d s\right|&\leq 2\alpha q \int_{\mathbb{T}^d}|\nabla\u_{\e}(x)|^2|\u_{\e}(x)|^{q-1}\d x\nonumber\\&=2\alpha q \int_{\mathbb{T}^d}|\nabla\u_{\e}(x)|^{2\left(\frac{r-1}{q-1}\right)}|\u_{\e}(x)|^{q-1}|\nabla\u_{\e}(x)|^{2\left(\frac{q-r}{q-1}\right)}\d x\nonumber\\&\leq \frac{\beta}{2}\int_{\mathbb{T}^d}|\nabla\u_{\e}(x)|^2|\u_{\e}(x)|^{r-1}\d x+\zeta\int_{\mathbb{T}^d}|\nabla\u_{\e}(x)|^2\d x,
	\end{align}
where $\zeta=(2\alpha q)^{\frac{r-1}{r-q}}\left(\frac{r-q}{r-1}\right)\left(\frac{2(q-1)}{\beta(r-1)}\right)^{\frac{q-1}{r-q}}$.  For $d=3$, we estimate $|(\B(\u_{\e}),\A\u_{\e})|$ using H\"older's and Young's inequalities as 
\begin{align}\label{275}
	|(\B(\u_{\e}),\A\u_{\e})|&\leq\||\u_{\e}||\nabla\u_{\e}|\|_{\H}\|\A\u_{\e}\|_{\H}\leq\frac{\mu }{4}\|\A\u_{\e}\|_{\H}^2+\frac{1}{\mu }\||\u_{\e}||\nabla\u_{\e}|\|_{\H}^2. 
\end{align}
For $r>3$, we  estimate the final term from \eqref{275} using H\"older's and Young's inequalities as 
\begin{align*}
	&	\int_{\mathbb{T}^d}|\u_{\e}(x)|^2|\nabla\u_{\e}(x)|^2\d x\nonumber\\&=\int_{\mathbb{T}^d}|\u_{\e}(x)|^2|\nabla\u_{\e}(x)|^{\frac{4}{r-1}}|\nabla\u_{\e}(x)|^{\frac{2(r-3)}{r-1}}\d x\nonumber\\&\leq\left(\int_{\mathbb{T}^d}|\u_{\e}(x)|^{r-1}|\nabla\u_{\e}(x)|^2\d x\right)^{\frac{2}{r-1}}\left(\int_{\mathbb{T}^d}|\nabla\u_{\e}(x)|^2\d x\right)^{\frac{r-3}{r-1}}\nonumber\\&\leq\frac{\beta\mu }{4}\left(\int_{\mathbb{T}^d}|\u_{\e}(x)|^{r-1}|\nabla\u_{\e}(x)|^2\d x\right)+\vartheta\left(\int_{\mathbb{T}^d}|\nabla\u_{\e}(x)|^2\d x\right),
\end{align*}
where $\vartheta=\frac{r-3}{r-1}\left(\frac{8}{\beta\mu (r-1)}\right)^{\frac{2}{r-3}}$. Therefore, from \eqref{275}, we have 
\begin{align}
	|(\B(\u_{\e}),\A\u_{\e})|&\leq\frac{\mu }{4}\|\A\u_{\e}\|_{\H}^2+\frac{\beta}{4}\||\u_{\e}|^{\frac{r-1}{2}}\nabla\u_{\e}\|_{\H}^2+\frac{\vartheta}{\mu}\|\nabla\u_{\e}\|_{\H}^2. 
\end{align}
We estimate $|(\F(\u_{\e}),\A\u_{\e})|$ using H\"older's and Young's inequalities, and Hypothesis \ref{hyp2} as 
\begin{align}\label{3p29}
	|(\F(\u_{\e}),\A\u_{\e})|&\leq\|\F(\u_{\e})\|_{\H}\|\A\u_{\e}\|_{\H}\leq\frac{\mu}{4}\|\A\u_{\e}\|_{\H}^2+\frac{1}{\mu}\|\F(\u_{\e})\|_{\H}^2\nonumber\\&\leq \frac{\mu}{4}\|\A\u_{\e}\|_{\H}^2+\frac{K_1}{\mu}(1+\|\u_{\e}\|_{\H}^2). 
\end{align}
Combining \eqref{3p25}-\eqref{3p29} and then substituting it in \eqref{3p24}, we deduce $\mathbb{P}$-a.s.,
\begin{align}
		&	\|\u_{\e}(t)\|_{\V}^2+\mu\int_0^t\|\A\u_{\e}(s)\|_{\H}^2\d s+\beta\int_0^t\||\u_{\e}(s)|^{\frac{r-1}{2}}\nabla\u_{\e}(s)\|_{\H}^2\d s\nonumber\\&\quad+8\beta\left[\frac{r-1}{(r+1)^2}\right]\int_0^t\|\nabla|\u_{\e}(s)|^{\frac{r+1}{2}}\|_{\H}^2\d s\nonumber\\&\leq\|\h\|_{\V}^2+2\left(\frac{\vartheta}{\mu}+\zeta\right)\int_0^t\|\nabla\u_{\e}(s)\|_{\H}^2\d s+\frac{2K_1}{\mu}\int_0^t\left(1+\|\u_{\e}(s)\|_{\H}^2\right)\d s\nonumber\\&\quad+\int_0^t\int_{\R_0}\|\boldsymbol{\sigma}^{\e}(\u_{\e}(s),z)\|_{\V}^2\pi(\d s,\d z)+2\int_0^t\int_{\R_0}(\nabla\boldsymbol{\sigma}^{\e}(\u_{\e}(s-),z),\nabla\u_{\e}(s-))\wi{\pi}(\d s,\d z).
\end{align}
Taking supremum over time $t\in[0,T]$ and expectation, and then  using calculations similar to \eqref{3p15}, we obtain 
\begin{align}\label{3p31}
&\E\left[\sup_{t\in[0,T]}	\|\u_{\e}(t)\|_{\V}^2+2\mu\int_0^T\|\A\u_{\e}(t)\|_{\H}^2\d t+2\beta\int_0^T\||\u_{\e}(t)|^{\frac{r-1}{2}}\nabla\u_{\e}(t)\|_{\H}^2\d t\right]	\nonumber\\&\leq\|\h\|_{\V}^2+4\left(\frac{\vartheta}{\mu}+\zeta\right)\E\left[\int_0^T\|\u_{\e}(t)\|_{\V}^2\d t\right]+\frac{4K_1}{\mu}\E\left[\int_0^T\left(1+\|\u_{\e}(t)\|_{\H}^2\right)\d t\right]\nonumber\\&\quad+C\E\left[\int_0^T\int_{\R_0}\|\boldsymbol{\sigma}^{\e}(\u_{\e}(t),z)\|_{\V}^2\lambda(\d z)\d t\right]\nonumber\\&\leq \|\h\|_{\V}^2+4\left(\frac{\vartheta}{\mu}+\zeta+C\right)\E\left[\int_0^T\|\u_{\e}(t)\|_{\V}^2\d t\right]+\frac{4K_1}{\mu}\E\left[\int_0^T\left(1+\|\u_{\e}(t)\|_{\H}^2\right)\d t\right]\nonumber\\&\leq C(1+\|\h\|_{\V}^2),
\end{align}
where we have used Hypothesis \ref{hyp3} and \eqref{3p17}. For $d=2$, we know that $(\B(\u_{\e}),\A\u_{\e})=0$ (Lemma 3.1, \cite{Te1}), so that the estimate \eqref{3p31} holds true for all $r\in[1,\infty)$. 

For $d=r=3$, we estimate $	|(\B(\u_{\e}),\A\u_{\e})|$, $	|(\widetilde{\mathcal{C}}(\u_{\e}),\A\u_{\e})|$ and $	|(\F(\u_{\e}),\A\u_{\e})|$ as 
\begin{align}
		|(\B(\u_{\e}),\A\u_{\e})|&\leq\||\u_{\e}||\nabla\u_{\e}|\|_{\H}\|\A\u_{\e}\|_{\H}\leq\frac{\theta\mu}{2}\|\A\u_{\e}\|_{\H}^2+\frac{1}{2\theta\mu }\||\u_{\e}||\nabla\u_{\e}|\|_{\H}^2, \\
		|(\widetilde{\mathcal{C}}(\u_{\e}),\A\u_{\e})|&\leq\frac{(1-\theta)\beta}{\alpha}\||\nabla\u_{\e}||\u_{\e}|\|_{\H}+\widetilde{\zeta}\|\nabla\u_{\e}\|_{\H}^2,\\
		|(\F(\u_{\e}),\A\u_{\e})|&\leq\|\F(\u_{\e})\|_{\H}\|\A\u_{\e}\|_{\H}\leq\frac{\theta\mu}{2}\|\A\u_{\e}\|_{\H}^2+\frac{1}{2\theta\mu}\|\F(\u_{\e})\|_{\H}^2,
\end{align}
for $0<\theta< 1$, where $\widetilde{\zeta}=\left(\frac{r-q}{r-1}\right)\left(\frac{\alpha(q-1)}{(1-\theta)\beta(r-1)}\right)^{\frac{q-1}{r-q}}$. Therefore, a calculation similar to \eqref{3p31} yields 
\begin{align}\label{3p34}
	&\E\left[\sup_{t\in[0,T]}	\|\u_{\e}(t)\|_{\V}^2+2(1-\theta)\mu\int_0^T\|\A\u_{\e}(t)\|_{\H}^2\d t+\left(2\theta\beta-\frac{1}{\theta\mu}\right)\int_0^T\||\u_{\e}(t)|^{\frac{r-1}{2}}\nabla\u_{\e}(t)\|_{\H}^2\d t\right]\nonumber\\&\leq C(\mu,\beta,T)(1+\|\h\|_{\V}^2),
\end{align}
and for $2\beta\mu>1$, the required result follows. 
\end{proof}

\begin{proposition}\label{prop3.9}
	For $\h\in\V$, under Hypotheses \ref{hyp2} and \ref{hyp3}, the family $\{\u_{\e}:0<\e\leq\e_0\}$ is tight in the space $\D([0,T];\V)$. 
\end{proposition}
	\begin{proof}
	By	Rellich-Kondrachov theorem, we know that the embedding of $\V\subset\H,$ is compact. Therefore, by Aldou's tightness criterion (see Theorem 1, \cite{DAl}), it suffices to show that:
	\begin{enumerate}
		\item [(i)] for any $0<\eta<1$, there exists an $L_{\eta}>0$ such that 
		\begin{align}
			\sup_{0<\e\leq\e_0}\mathbb{P}\left\{\sup_{0\leq t\leq T}\|\u_{\e}(t)\|_{\V}>L_{\eta}\right\}<\eta;
		\end{align}
	\item [(ii)] for any stopping time $0\leq \tau^{\e}\leq T$ with respect to the natural filtration generated by $\{\u_{\e}(s):s\leq t\}$, and any $\eta>0$, 
	\begin{align}\label{336}
		\lim_{\delta\to 0}\sup_{0<\e\leq\e_0}\mathbb{P}\left\{\|\u_{\e}(\tau^{\e}+\delta)-\u_{\e}(\tau^{\e})\|_{\H}>\eta\right\}=0,
	\end{align}
where we have set $\tau^{\e}+\delta:=T\wedge(\tau^{\e}+\delta)$.
	\end{enumerate}

For any $L>0$, note  that \eqref{323} implies 
\begin{align}\label{3p37}
		\sup_{0<\e\leq\e_0}\mathbb{P}\left\{\sup_{0\leq t\leq T}\|\u_{\e}(t)\|_{\V}>L\right\}&\leq \frac{1}{L^2}\sup_{0<\e\leq\e_0}\E\left[\sup_{0\leq t\leq T}\|\u_{\e}(t)\|_{\V}^2\right]\leq\frac{C}{L^2}(1+\|\h\|_{\V}^2). 
\end{align}
Therefore given any $0<\eta<1$, we can find an $L$ large enough such that the right hand side of \eqref{3p37} can be made less than $\eta$, so that (i) is satisfied.

Let us now prove (ii). By an application of Markov's inequality, it is enough to show that 
\begin{align}\label{338}
	\E\left[\|\u_{\e}(\tau^{\e}+\delta)-\u_{\e}(\tau^{\e})\|_{\H}^{\xi}\right]\leq C\delta^{\theta},
\end{align}
for some $\xi,\theta>0$ and a $C>0$. From \eqref{3p1}, we have 
	\begin{align}
	\u_{\e}(t)&=\boldsymbol{h}-\mu\int_0^t\A\u_{\e}(s)\d s-\int_0^t\B(\u_{\e}(s))\d s-\alpha\int_0^t\widetilde{\mathcal{C}}(\u_{\e}(s))\d s-\beta\int_0^t\mathcal{C}(\u_{\e}(s))\d s\nonumber\\&\quad+\int_0^t\F(\u_{\e}(s))\d s+\int_0^t\int_{\R_0}\boldsymbol{\sigma}^{i,\e}(\u_{\e}(s-),z)\wi{\pi}(\d s,\d z)\nonumber\\&=\h+\sum\limits_{i=1}^6J_i^{\e}(t),
\end{align}
for all $t\in[0,T]$ in $\V'+\wi\L^{\frac{r+1}{r}}$. For $J_1^{\e}$, we use \eqref{323} to estimate as 
\begin{align}\label{340}
\E\left[\|J_1^{\e}(\tau^{\e}+\delta)-J_1^{\e}(\tau^{\e})\|_{\H}\right]&\leq \mu\E\left[\int_{\tau^{\e}}^{\tau^{\e}+\delta}\|\A\u_{\e}(s)\|_{\H}\d s\right]\leq \mu\delta^{\frac{1}{2}}\E\left[\left(\int_{\tau^{\e}}^{\tau^{\e}+\delta}\|\A\u_{\e}(s)\|_{\H}^2\d s\right)^{\frac{1}{2}}\right]\nonumber\\&\leq C\delta^{\frac{1}{2}}(1+\|\h\|_{\V}),
\end{align}
so that \eqref{338} is satisfied for $\xi=1$ and $\theta=\frac{1}{2}$. For $r\geq 3$, using the fact that $|x|\leq 1+|x|^{\frac{r-1}{2}}$, for all $x\in\R$,  we get 
\begin{align}
\E\left[\|J_2^{\e}(\tau^{\e}+\delta)-J_2^{\e}(\tau^{\e})\|_{\H}\right]&\leq\E\left[\int_{\tau^{\e}}^{\tau^{\e}+\delta}\|\B(\u_{\e}(s))\|_{\H}\d s\right]\leq\delta^{\frac{1}{2}}\E\left[\left(\int_{\tau^{\e}}^{\tau^{\e}+\delta}\|\B(\u_{\e}(s))\|_{\H}^2\d s\right)^{\frac{1}{2}}\right]\nonumber\\&\leq \delta^{\frac{1}{2}}\left\{\E\left[\int_0^T\||\u_{\e}(t)||\nabla\u_{\e}(t)|\|_{\H}^2\d t\right]\right\}^{\frac{1}{2}}\nonumber\\&\leq C\delta^{\frac{1}{2}}\left\{\E\left[\int_0^T\left(\|\nabla\u_{\e}(t)\|_{\H}^2+\||\u_{\e}(t)|^{\frac{r-1}{2}}|\nabla\u_{\e}(t)|\|_{\H}^2\right)\d t\right]\right\}^{\frac{1}{2}}\nonumber\\&\leq C\delta^{\frac{1}{2}}\left(1+\|\h\|_{\V}\right),
\end{align}
therefore \eqref{338} is satisfied for $\xi=1$ and $\theta=\frac{1}{2}$.  For $d=2$ and $r\in[1,3]$, we use H\"older's and Agmon's inequalities  to obtain 
\begin{align}
	\E\left[\|J_2^{\e}(\tau^{\e}+\delta)-J_2^{\e}(\tau^{\e})\|_{\H}\right]&\leq\E\left[\int_{\tau^{\e}}^{\tau^{\e}+\delta}\|\u_{\e}(s)\|_{\wi\L^{\infty}}\|\nabla\u_{\e}(s)\|_{\H}\d s\right]\nonumber\\&\leq \E\left[\int_{\tau^{\e}}^{\tau^{\e}+\delta}\|\u_{\e}(s)\|_{\H}^{\frac{1}{2}}\|\A\u_{\e}(s)\|_{\H}^{\frac{1}{2}}\|\nabla\u_{\e}(s)\|_{\H}\d s\right]\nonumber\\&\leq \delta^{\frac{3}{4}}\left\{\E\left[\sup_{s\in[\tau^{\e},\tau^{\e}+\delta]}\|\u_{\e}(s)\|_{\H}^2\right]\right\}^{\frac{1}{4}} \left\{\E\left[\sup_{s\in[\tau^{\e},\tau^{\e}+\delta]}\|\u_{\e}(s)\|_{\V}^2\right]\right\}^{\frac{1}{2}}\nonumber\\&\quad\times\left\{\E\left[\int_{\tau^{\e}}^{\tau^{\e}+\delta}\|\A\u_{\e}(s)\|_{\H}^2\d s\right]\right\}^{\frac{1}{4}}\nonumber\\&\leq C\delta^{\frac{3}{4}}(1+\|\h\|_{\V}),
\end{align}
and the condition \eqref{338} follows with  $\xi=1$ and $\theta=\frac{3}{4}$. Let us now consider 
\begin{align}\label{346}
	&\E\left[\|J_3^{\e}(\tau^{\e}+\delta)-J_3^{\e}(\tau^{\e})\|_{\H}\right]\nonumber\\&=\E\left[\int_{\tau^{\e}}^{\tau^{\e}+\delta}\|\mathcal{C}(\u_{\e}(s))\|_{\H}\d s\right]\leq\E\left[\int_{\tau^{\e}}^{\tau^{\e}+\delta}\|\u_{\e}(s)\|_{\wi\L^{2r}}^r\d s\right]\nonumber\\&\leq\E\left[\int_{\tau^{\e}}^{\tau^{\e}+\delta}\|\u_{\e}(s)\|_{\wi\L^{r+1}}^{\frac{r+3}{4}}\|\u_{\e}(s)\|_{\wi\L^{3(r+1)}}^{\frac{3(r-1)}{4}}\d s\right]\nonumber\\&\leq \delta^{\frac{1}{r+1}}\left\{\E\left[\int_{\tau^{\e}}^{\tau^{\e}+\delta}\|\u_{\e}(s)\|_{\wi\L^{r+1}}^{r+1}\d s\right]\right\}^{\frac{r+3}{4(r+1)}}\left\{\E\left[\int_{\tau^{\e}}^{\tau^{\e}+\delta}\|\u_{\e}(s)\|_{\wi\L^{3(r+1)}}^{r+1}\d s\right]\right\}^{\frac{3(r-1)}{4(r+1)}}\nonumber\\&\leq C\delta^{\frac{1}{r+1}}\left(1+\|\h\|_{\V}\right),
\end{align}
where we have used the interpolation inequality. Therefore, $J_3^{\e}$ satisfies condition \eqref{338} with $\xi=1$ and $\theta=\frac{1}{r+1}$. Using the interpolation and H\"older's inequalities and \eqref{346}, we find 
\begin{align}
	&\E\left[\|J_4^{\e}(\tau^{\e}+\delta)-J_4^{\e}(\tau^{\e})\|_{\H}\right]\nonumber\\&=\E\left[\int_{\tau^{\e}}^{\tau^{\e}+\delta}\|\widetilde{\mathcal{C}}(\u_{\e}(s))\|_{\H}\d s\right]\leq\E\left[\int_{\tau^{\e}}^{\tau^{\e}+\delta}\|\u_{\e}(s)\|_{\wi\L^{2q}}^q\d s\right]\nonumber\\&\leq\E\left[\sup_{t\in[0,T]}\|\u_{\e}(t)\|_{\H}^{\frac{r-q}{r-1}}\left(\int_{\tau^{\e}}^{\tau^{\e}+\delta}\|\u_{\e}(s)\|_{\wi\L^{2r}}^{\frac{r(q-1)}{r-1}}\d s\right)\right]\nonumber\\&\leq\delta^{\frac{r-q}{r-1}}\left\{\E\left[\sup_{t\in[0,T]}\|\u(t)\|_{\H}\right]\right\}^{\frac{r-q}{r-1}}\left\{\E\left[\int_0^T\|\u_{\e}(t)\|_{\wi\L^{2r}}^r\d t\right]\right\}^{\frac{q-1}{r-1}}\nonumber\\&\leq C\delta^{\frac{r(r+1)-(qr+1)}{(r-1)(r+1)}}\left(1+\|\h\|_{\V}\right),
	\end{align}
so that $J_4^{\e}$ satisfies condition \eqref{338} with $\xi=1$ and $\theta={\frac{r(r+1)-(qr+1)}{(r-1)(r+1)}}$. Using Hypothesis \ref{hyp2}, we infer 
\begin{align}
	&\E\left[\|J_5^{\e}(\tau^{\e}+\delta)-J_5^{\e}(\tau^{\e})\|_{\H}\right]\nonumber\\&=\E\left[\int_{\tau^{\e}}^{\tau^{\e}+\delta}\|\F(\u_{\e}(s))\|_{\H}\d s\right]\leq \delta\left\{\E\left[\sup_{s\in[\tau^{\e},\tau^{\e}+\delta]}\|\F(\u_{\e}(s))\|_{\H}^2\right]\right\}^{\frac{1}{2}}\nonumber\\&\leq K_1\delta \left\{\E\left[\sup_{s\in[\tau^{\e},\tau^{\e}+\delta]}(1+\|\u_{\e}(s)\|_{\H}^2)\right]\right\}^{\frac{1}{2}}\leq K_1\delta\left(1+\|\h\|_{\V}\right),
\end{align}
so that the condition \eqref{338} is satisfied with $\xi=\theta=1$. Using It\^o's isometry and Hypothesis \ref{hyp2}, we have 
\begin{align}\label{345}
	&\E\left[\|J_6^{\e}(\tau^{\e}+\delta)-J_6^{\e}(\tau^{\e})\|_{\H}^2\right]\nonumber\\&=\E\left[\left\|\int_{\tau^{\e}}^{\tau^{\e}+\delta}\int_{\R_0}\boldsymbol{\sigma}^{\e}(\u_{\e}(s-),z)\wi{\pi}(\d s,\d z)\right\|_{\H}^2\right]\nonumber\\&=\E\left[\int_{\tau^{\e}}^{\tau^{\e}+\delta}\int_{\R_0}\|\boldsymbol{\sigma}^{\e}(\u_{\e}(s),z)\|_{\H}^2\lambda(\d z)\d s\right]\nonumber\\&\leq K_1\delta\E\left[\sup_{s\in[\tau^{\e},\tau^{\e}+\delta]}(1+\|\u_{\e}(s)\|_{\H}^2)\right]\leq C\delta\left(1+\|\h\|_{\V}^2\right). 
\end{align}
Thus the condition \eqref{338} is satisfied with $\xi=2$ and $\theta=1$. Therefore combining \eqref{340}-\eqref{345} and using Markov's inequality, one can get 
\begin{align}
&	\sup_{0<\e\leq\e_0}\mathbb{P}\left\{\|\u_{\e}(\tau^{\e}+\delta)-\u_{\e}(\tau^{\e})\|_{\H}>\eta\right\}\nonumber\\&\leq\frac{1}{\eta^{\theta}}\E\left[\|\u_{\e}(\tau^{\e}+\delta)-\u_{\e}(\tau^{\e})\|_{\H}^{\theta}\right]\leq\frac{C\delta^{\theta}}{\eta^{\theta}}\to 0\ \text{ as }\ \delta\to 0
\end{align}
for any $\eta>0$. Hence (ii) is verified and the proof is completed. 
	\end{proof}

\subsection{Weak convergence} Let  $$\nu_{\e},\nu\ \text{ denote the laws of }\  \u_{\e}\ \text{ and }\ \u$$ on the spaces  $\D([0,T];\H)$ and $\C([0,T];\H)$, respectively. We prove  the weak convergence by two steps. We first establish  the weak convergence in Theorem \ref{thm310} under stronger conditions, and then we remove the extra assumptions and obtain  the general convergence result in Theorem \ref{thm3.11}.

\begin{theorem}\label{thm310}
	Let $\h\in\V$. Under Hypotheses \ref{hyp2}, \ref{hyp3} (H.3) and \ref{hyp4}, for any $T>0$, $\nu_{\e}$ converges weakly to $\nu$, as $\e\to 0$, on the space $\D([0,T];\H)$  equipped with the Skorokhod topology.
\end{theorem}
\begin{proof}
	From Proposition \ref{prop3.9}, we infer that the family $\{\nu_{\e}:0<\e\leq\e_0\}$ is tight in $\D([0,T];\H)$. Let $\nu_0$  be the weak limit of any convergent subsequence $\{\nu_{\e_n}\}_{n\geq 1}$. Our aim is to  show that $\nu_0=\nu$. The  proof is divided into the following three steps:  \begin{enumerate} \item  [{\bf S1.}]  $\nu_0$ is supported on the space $\C ([0, T ];\H)$, \item  [{\bf S2.}]  $\nu_0$ is a solution of a martingale problem, \item  [{\bf S3.}] $\nu_0$ is the law of a weak solution of stochastic CBFeD equations \eqref{2p11},
		\end{enumerate} 
	and one can complete the proof. 
	\vskip 0.2 cm 
	{\bf S1.} For $\eta>0,M>0$, we find 
	\begin{align}\label{3p47}
	&	\mathbb{P}\left\{\sup_{0\leq t\leq T}\|\u_{\e}(t)-\u_{\e}(t-)\|_{\H}\geq\eta \right\}\nonumber\\&\leq \mathbb{P}\left\{\sup_{0\leq t\leq T}\sup_{z\in\R_0}\|\sigma^{\e}(\u_{\e}(t-),z)\|_{\H}\geq\eta\right\}\nonumber\\&\leq \mathbb{P}\left\{\sup_{0\leq t\leq T}\sup_{z\in\R_0}\|\sigma^{\e}(\u_{\e}(t-),z)\|_{\H}\geq\eta,\sup_{0\leq t\leq T}\|\u_{\e}(t)\|_{\H}\leq M\right\}+\mathbb{P}\left\{\sup_{0\leq t\leq T}\|\u_{\e}(t)\|_{\H}>M\right\}\nonumber\\&\leq \mathbb{P}\left\{\sup_{\|\u\|_{\H}\leq M}\sup_{z\in\R_0}\|\sigma^{\e}(\u,z)\|_{\H}>\eta\right\}+\frac{1}{M^2}\sup_{0<\e\leq\e_0}\E\left[\sup_{0\leq t\leq T}\|\u_{\e}(t)\|_{\H}^2\right].
	\end{align}
Making use of  \eqref{3p7} and \eqref{3p17}, we first let $\e\to 0$ and then $M\to\infty$ in \eqref{3p47}, we obtain 
\begin{align}
	\sup_{0\leq t\leq T}\|\u_{\e}(t)-\u_{\e}(t-)\|_{\H} \xrightarrow{p} 0\ \text{ as }\ \e\to 0. 
\end{align}
Therefore, it follows from Theorem 13.4, \cite{PBi} that $\nu_0$ is supported on the space $\C ([0, T ];\H)$. As a consequence, the finite-dimensional distributions of $\nu_{\e_n}$ converge to that of $\nu_0$.

\vskip 0.2 cm 
	{\bf S2.} For $j,k\in\mathbb{N}$, let us take $f(\x)=(\x,\boldsymbol{e}_k)(\x,\boldsymbol{e}_j), \ \x\in\H$. The gradient of $f$ is denoted by $\nabla f$  and the operator associated with the second derivatives of $f$ is represented by $f''$. Then, $\nabla f$ and $f''$ are given by 
	\begin{align}
		(\nabla f(\x),\boldsymbol{h})&=(\boldsymbol{h},(\x,\boldsymbol{e}_j)\boldsymbol{e}_k+(\x,\boldsymbol{e}_k)\boldsymbol{e}_j),\\
		[f''(\x)(\boldsymbol{h}\otimes\boldsymbol{k})]&=(\boldsymbol{k},\boldsymbol{e}_j)(\boldsymbol{h},\boldsymbol{e}_k)
		+(\boldsymbol{k},\boldsymbol{e}_k)(\boldsymbol{h},\boldsymbol{e}_j),
	\end{align}
for all $\boldsymbol{h},\boldsymbol{k}\in\H$. 
Let us set 
\begin{align}
	\mathscr{L}^{\e}f(\x):&=-(\mu\A\x+\B(\x)+\beta\mathcal{C}(\x)+\F(\x),\nabla f(\x))\nonumber\\&\quad+\int_{\R_0}[f(\x+\boldsymbol{\sigma}^{\e}(\x,z))-f(\x)-(\boldsymbol{\sigma}^{\e}(\x,z),\nabla f(\x))]\lambda(\d z),\label{365}\\
	\mathscr{L}f(\x):&=-(\mu\A\x+\B(\x)+\beta\mathcal{C}(\x)+\F(\x),\nabla f(\x))+\frac{1}{2}[f''(\x)(\sigma(\x)\otimes\sigma(\x))].\label{366}
\end{align}
An application of  It\^o's formula yields 
\begin{align}
&	f(\u_{\e}(t))-f(\h)-\int_0^t\mathscr{L}^{\e}f(\u_{\e}(s))\d s\nonumber\\&=\int_0^t\int_{\R_0}\left[f(\u_{\e}(s-)+\sigma^{\e}(\u_{\e}(s-),z))-f(\u_{\e}(s-))\right]\wi{\pi}(\d s, \d z)
\end{align}
is a martingale. Let us denote by ${X}_t(\omega):=\omega(t), \ \omega\in\D([0,T];\H),$ the coordinate process on $\D([0,T];\H)$. By the above martingale property, for any $m\in\N$, $0 \leq s_0 < s_1 <\cdots<
s_m \leq  s < t$  and $ f_0, f_1, \ldots, f_m \in\C_b(\H)$ (the collection of bounded continuous functions on
$\H$), it holds that
\begin{align}\label{368}
	\E^{\nu_{\e}}\left[\left(f(X_t)-f(X_s)-\int_s^t\mathscr{L}^{\e}f(X_r)\d r\right)f_0(X_{s_0})\cdots f_m(X_{s_m})\right]=0. 
\end{align}
Let 
\begin{align}\label{369}
	\G_{\e}(\x):=\left|\int_{\R_0}(\sigma^{\e}(\x,z),\boldsymbol{e}_k)(\sigma^{\e}(\x,z),\boldsymbol{e}_j)\lambda(\d z)-(\sigma(\x),\boldsymbol{e}_k)(\sigma(\x),\boldsymbol{e}_j)\right|,
\end{align}
for $\x\in\H$. From \eqref{365} and \eqref{366}, we infer that 
\begin{align}
	\left|\mathscr{L}^{\e}f(X_r)-\mathscr{L}f(X_r)\right|=\G_{\e}(X_r). 
\end{align}
We claim that 
\begin{align}\label{371}
	\lim\limits_{n\to\infty}\E^{\mu_{\e_n}}\left[\int_s^t\left|\mathscr{L}^{\e_n}f(X_r)-\mathscr{L}f(X_r)\right|\d r\right]=0.
\end{align}
It should be noted that 
\begin{align}
	\E^{\mu_{\e_n}}\left[\int_s^t\left|\mathscr{L}^{\e_n}f(X_r)-\mathscr{L}f(X_r)\right|\d r\right]&=\E^{\mu_{\e_n}}\left[\int_s^t\G_{\e_n}(X_r)\d r\right]=\int_s^t\E\left[\G_{\e_n}(\u_{\e_n}(r))\right]\d r,\label{372}\\
	\sup_{0<\e\leq\e_0}\G_{\e}(\x)&\leq C\left(1+\|\x\|_{\H}^2\right). \label{373}
\end{align}
By the dominated convergence theorem and \eqref{3p12}, in order to show \eqref{371}, it suffices to prove that for every $r\in[0,T]$, 
\begin{align}\label{374}
	\lim_{n\to\infty}\E\left[\G_{\e_n}(\u_{\e_n}(r))\right]=0. 
\end{align}
Let us now take any $r\in[0, T ]$ and fix it. Since the finite-dimensional distributions of $\nu^{\e_n}$ converge weakly to that of $\nu_0$, by Skorohod's representation theorem (along a subsequence), we can assume that $\u_{\e_n}(r)$ converges almost surely to an $\H$-valued random variable $\u_0$. As $\{\|\u_{\e_n}(r)\|_{\H}^2\}_{n\geq 1}$ is uniformly integrable (see \eqref{3p12}), we can deduce the existence of a $\u_0\in\mathrm{L}^2(\Omega;\H)$ such that (Theorem 13.7, \cite{DW})
\begin{align}\label{375}
	\lim_{n\to\infty}\E\left[\|\u_{\e_n}(r)-\u_0\|_{\H}^2\right]=0.
\end{align}
By the dominated convergence theorem, it follows from \eqref{372} and \eqref{373} that
\begin{align}
	\lim\limits_{n\to\infty}\E\left[\G_{\e_n}(\u_0)\right]=0. 
\end{align}
Therefore, in order to show \eqref{374}, it is suffices to show that 
\begin{align}\label{377}
	\lim\limits_{n\to\infty}\E\left[\left|\G_{\e_n}(\u_{\e_n}(r))-\G_{\e_n}(\u_0)\right|\right]=0. 
\end{align}
By the definition of $\G_{\e}(\cdot)$ in \eqref{369}, we have 
\begin{align}
&\E\left[\left|\G_{\e_n}(\u_{\e_n}(r))-\G_{\e_n}(\u_0)\right|\right]\nonumber\\&\leq \E\left[\left|\int_{\R_0}(\sigma^{\e_n}(\u_{\e_n}(r),z),\boldsymbol{e}_k)(\sigma^{\e_n}(\u_{\e_n}(r),z),\boldsymbol{e}_j)\lambda(\d z)\right.\right.\nonumber\\&\qquad\left.\left.-\int_{\R_0}(\sigma^{\e_n}(\u_0,z),\boldsymbol{e}_k)(\sigma^{\e_n}(\u_0,z),\boldsymbol{e}_j)\lambda(\d z)\right|\right]\nonumber\\&\quad+\E\left[\left|(\sigma(\u_{\e_n}),\boldsymbol{e}_k)(\sigma(\u_{\e_n}),\boldsymbol{e}_j)-(\sigma(\u_0),\boldsymbol{e}_k)(\sigma(\u_0),\boldsymbol{e}_j)\right|\right]\nonumber\\&:=I_1^n+I_2^n. 
\end{align}
In view of \eqref{32} and \eqref{34}, we obtain 
\begin{align}
	I_1^n&\leq  \E\left[\int_{\R_0}\left|(\sigma^{\e_n}(\u_{\e_n}(r),z),\boldsymbol{e}_k)(\sigma^{\e_n}(\u_{\e_n}(r),z)-\sigma^{\e_n}(\u_0,z),\boldsymbol{e}_j)\right|\lambda(\d z)\right]\nonumber\\&\quad+\E\left[\int_{\R_0}\left|(\sigma^{\e_n}(\u_{\e_n}(r),z)-\sigma^{\e_n}(\u_0,z),\boldsymbol{e}_k)(\sigma^{\e_n}(\u_0,z),\boldsymbol{e}_j)\right|\lambda(\d z)\right]\nonumber\\&\leq\left\{\E\left[\int_{\R_0}\|\sigma^{\e_n}(\u_{\e_n}(r),z)\|_{\H}^2\lambda(\d z)\right]\right\}^{\frac{1}{2}}\left\{\E\left[\int_{\R_0}\|\sigma^{\e_n}(\u_{\e_n}(r),z)-\sigma^{\e_n}(\u_0,z)\|_{\H}^2\lambda(\d z)\right]\right\}^{\frac{1}{2}}\nonumber\\&\quad+\left\{\E\left[\int_{\R_0}\|\sigma^{\e_n}(\u_{\e_n}(r),z)-\sigma^{\e_n}(\u_0,z)\|_{\H}^2\lambda(\d z)\right]\right\}^{\frac{1}{2}}\left\{\E\left[\int_{\R_0}\|\sigma^{\e_n}(\u_0,z)\|_{\H}^2\lambda(\d z)\right]\right\}^{\frac{1}{2}}\nonumber\\&\leq C\left\{\left(1+\E\left[\|\u_0\|_{\H}^2\right]\right)^{\frac{1}{2}}+\left(1+\sup\limits_{\e_n}\E\left[\|\u_{\e_n}(r)\|_{\H}^2\right]\right)^{\frac{1}{2}}\right\}\left\{\E\left[\|\u_{\e_n}(r)-\u^0\|_{\H}^2\right]\right\}^{\frac{1}{2}}.
\end{align}
Taking in account of \eqref{3p12} and \eqref{375}, we infer that $I_1^n\to 0$  as $n\to\infty$. In a similar way, one can show that $I_2^n\to 0$ as $n\to\infty$, so that \eqref{377} follows. Therefore, equation \eqref{371} is validated.

Let us now show that 
\begin{align}\label{380}
	\M_{k,j}(t):=f(X_t)-f(\h)-\int_0^t\mathscr{L}f(X_r)\d r
\end{align}
is a martingale under $\nu_0$. This is equivalent to proving that 
\begin{align}
		\E^{\nu_0}\left[\left(f(X_t)-f(X_s)-\int_s^t\mathscr{L}f(X_r)\d r\right)f_0(X_{s_0})\cdots f_m(X_{s_m})\right]=0. 
\end{align}
Since the finite-dimensional distributions of $\nu^{\e_n}$ converge to that of $\nu_0$, using the fact that  $\|f(\x)\|_{\H}\leq\|\x\|_{\H}^2$ and the uniform energy estimate \eqref{3p12}, it follows from \cite[Theorem 1.6.8]{RD} (or \cite[Lemma 15]{SSS}) that 
\begin{align}\label{382}
	\E^{\nu_0}\left[f(X_t)f_0(X_{s_0})\cdots f_m(X_{s_m})\right]=\lim\limits_{n\to\infty}\E^{\nu_n}\left[f(X_t)f_0(X_{s_0})\cdots f_m(X_{s_m})\right]. 
\end{align}
Let us now show that 
\begin{align}\label{383}
&	\E^{\nu_0}\left[\left(\int_s^t\mathscr{L}f(X_r)\d r \right)f_0(X_{s_0})\cdots f_m(X_{s_m})\right]\nonumber\\&=\lim\limits_{n\to\infty}\E^{\nu_n}\left[\left(\int_s^t\mathscr{L}f(X_r)\d r \right)f_0(X_{s_0})\cdots f_m(X_{s_m})\right]. 
\end{align}
We use \cite[Lemma 15]{SSS} to obtain the above result. One needs to show that 
\begin{align}\label{384}
\sup\limits_{n\geq 1}	\E^{\nu_n}\left[\left|\int_s^t\mathscr{L}f(X_r)\d r \right|^{1+\epsilon}\right]\leq C,
\end{align}
for some $\epsilon>0$. It can be easily seen that 
\begin{align}\label{385}
|(\A\x,\nabla f(\x))|&=|(\x,\boldsymbol{e}_j)(\A\x,\boldsymbol{e}_k)+(\x,\boldsymbol{e}_k)(\A\x,\boldsymbol{e}_j)|\nonumber\\&\leq |(\x,\boldsymbol{e}_j)||(\x,\lambda_k\boldsymbol{e}_k)|+|(\x,\boldsymbol{e}_k)||(\x,\lambda_j\boldsymbol{e}_j)|\nonumber\\&\leq(\lambda_k+\lambda_j)\|\x\|_{\H}^2. 
\end{align}
Using H\"older's and Sobolev's inequalities, we estimate $|\langle\B(\x),\nabla f(\x)\rangle|$ for $s>\frac{d}{2}+1$ as 
\begin{align}\label{386}
|\langle\B(\x),\nabla f(\x)\rangle|&=|(\x,\boldsymbol{e}_j)\langle\B(\x),\boldsymbol{e}_k\rangle+(\x,\boldsymbol{e}_k)\langle\B(\x),\boldsymbol{e}_j\rangle|\nonumber\\&\leq |(\x,\boldsymbol{e}_j)||\langle\B(\x,\boldsymbol{e}_k),\x\rangle|+|(\x,\boldsymbol{e}_k)||\langle\B(\x,\boldsymbol{e}_j),\x\rangle|\nonumber\\&\leq\left( \|\nabla\boldsymbol{e}_k\|_{\wi\L^{\infty}}+\|\nabla\boldsymbol{e}_j\|_{\wi\L^{\infty}}\right)\|\x\|_{\H}^3\leq\left(\|\A^{\frac{s}{2}}\boldsymbol{e}_k\|_{\H}^2+\|\A^{\frac{s}{2}}\boldsymbol{e}_j\|_{\H}^2\right)\|\x\|_{\H}^3\nonumber\\&\leq \left(\lambda_k^{\frac{s}{2}}+\lambda_j^{\frac{s}{2}}\right)\|\x\|_{\H}^3. 
\end{align}
Using H\"older's, Sobolev's and  interpolation inequalities, we estimate $|\langle\mathcal{C}(\x),\nabla f(\x)\rangle|$ for $s>\frac{d}{2}+1$  for $s-1>\frac{d}{2}$ as 
\begin{align}\label{387}
	|\langle\mathcal{C}(\x),\nabla f(\x)\rangle|&=|(\x,\boldsymbol{e}_j)\langle\mathcal{C}(\x),\boldsymbol{e}_k\rangle+(\x,\boldsymbol{e}_k)\langle\mathcal{C}(\x),\boldsymbol{e}_j\rangle|\nonumber\\&\leq\left(\|\boldsymbol{e}_k\|_{\wi\L^{\infty}}+\|\boldsymbol{e}_j\|_{\wi\L^{\infty}}\right)\|\x\|_{\H}\|\x\|_{\wi\L^r}^r\nonumber\\&\leq \left(\|\A^{\frac{s-1}{2}}\boldsymbol{e}_k\|_{\H}+\|\A^{\frac{s-1}{2}}\boldsymbol{e}_j\|_{\H}\right)\|\x\|_{\H}^{\frac{r+1}{r-1}}\|\x\|_{\wi\L^{r+1}}^{\frac{(r-2)(r+1)}{r-1}}\nonumber\\&\leq \left(\lambda_k^{\frac{s-1}{2}}+\lambda_j^{\frac{s-1}{2}}\right)\|\x\|_{\H}^{\frac{r+1}{r-1}}\|\x\|_{\wi\L^{r+1}}^{\frac{(r-2)(r+1)}{r-1}},
\end{align} 
for all $\x\in\wi\L^{r+1}$ and $r\in(2,\infty)$. A calculation similar to \eqref{387} yields  for all $\x\in\wi\L^{r+1}$ and $q\in[2,\infty)$
\begin{align}
	|\langle\widetilde{\mathcal{C}}(\x),\nabla f(\x)\rangle|&\leq\left(\|\boldsymbol{e}_k\|_{\wi\L^{\infty}}+\|\boldsymbol{e}_j\|_{\wi\L^{\infty}}\right)\|\x\|_{\H}\|\x\|_{\wi\L^q}^q\nonumber\\&\leq\left(\lambda_k^{\frac{s-1}{2}}+\lambda_j^{\frac{s-1}{2}}\right)\|\x\|_{\H}^{\frac{3r+1-2q}{r-1}}\|\x\|_{\wi\L^{r+1}}^{\frac{(r+1)(q-2)}{r-1}}.
\end{align}
For $q<r\in[1,2]$, one can estimate for all $\x\in\H$  $$|\langle\mathcal{C}(\x),\nabla f(\x)\rangle|\leq C\left(\lambda_k^{\frac{s-1}{2}}+\lambda_j^{\frac{s-1}{2}}\right)\|\x\|_{\H}^{r+1}.$$ Finally, we estimate $|[f''(\x)(\sigma(\x)\otimes\sigma(\x))]|$ as 
\begin{align}\label{388}
	|[f''(\x)(\sigma(\x)\otimes\sigma(\x))]|&=2|(\sigma(\x),\boldsymbol{e}_j)(\sigma(\x),\boldsymbol{e}_k)|\leq 2\|\sigma(\x)\|_{\H}^2\leq 2C(1+\|\x\|_{\H}^2).
\end{align}
For $q<r\in(2,\infty)$, combining \eqref{385}-\eqref{388}, one can deduce 
\begin{align}
	&\E^{\nu_n}\left[\left|\int_s^t\mathscr{L}f(X_r)\d r \right|^{1+\epsilon}\right]\nonumber\\&=\E\left[\left|\int_s^t\mathscr{L}f(\u_{\e_n}(r))\d r \right|^{1+\epsilon}\right]\nonumber\\&\leq C\left\{\E\left[\left|\int_s^t\langle\A\u_{\e_n}(r),\nabla f(\u_{\e_n}(r))\rangle\d r \right|^{1+\epsilon}\right]+\E\left[\left|\int_s^t\langle\B(\u_{\e_n}(r)),\nabla f(\u_{\e_n}(r))\rangle\d r \right|^{1+\epsilon}\right]\right.\nonumber\\&\quad+\left.\E\left[\left|\int_s^t\langle\mathcal{C}(\u_{\e_n}(r)),\nabla f(\u_{\e_n}(r))\rangle\d r \right|^{1+\epsilon}\right]+\E\left[\left|\int_s^t\langle\widetilde{\mathcal{C}}(\u_{\e_n}(r)),\nabla f(\u_{\e_n}(r))\rangle\d r \right|^{1+\epsilon}\right]\right.\nonumber\\&\left.\quad+\E\left[\left|\int_s^t[f''(\u_{\e_n}(r))(\sigma(\u_{\e_n}(r))\otimes\sigma(\u_{\e_n}(r)))]\d r \right|^{1+\epsilon}\right]\right\}\nonumber\\&\leq C\left\{T^{1+\epsilon}\E\left[\sup_{t\in[0,T]}\|\u_{\e_n}(t)\|_{\H}^{2(1+\epsilon)}\right]+T^{1+\epsilon}\E\left[\sup_{t\in[0,T]}\|\u_{\e_n}(t)\|_{\H}^{3(1+\epsilon)}\right]\right.\nonumber\\&\quad\left.+T^{\frac{1+\epsilon}{r-1}}\left(\E\left[\left(\int_0^T\|\u_{\e_n}(t)\|_{\wi\L^{r+1}}^{r+1}\d t\right)^{1+\epsilon}\right]\right)^{\frac{r-2}{r-1}}\left(\E\left[\sup\limits_{t\in[0,T]}\|\u_{\e_n}(t)\|_{\H}^{(r+1)(1+\epsilon)}\right]\right)^{\frac{1}{r-1}}\right.\nonumber\\&\quad\left.+T^{\frac{(r-q+1)(1+\epsilon)}{r-1}}\left(\E\left[\left(\int_0^T\|\u_{\e_n}(t)\|_{\wi\L^{r+1}}^{r+1}\d t\right)^{1+\epsilon}\right]\right)^{\frac{q-2}{r-1}}\left(\E\left[\sup\limits_{t\in[0,T]}\|\u_{\e_n}(t)\|_{\H}^{\frac{(3r+1-2q)(1+\epsilon)}{r-q+1}}\right]\right)^{\frac{r-q+1}{r-1}}\right.\nonumber\\&\quad\left.+T^{1+\epsilon}\left(1+\E\left[\sup_{t\in[0,T]}\|\u_{\e_n}(t)\|_{\H}^{2(1+\epsilon)}\right]\right)\right\}\leq C,
\end{align}
by using \eqref{3p12}, so that \eqref{384} follows. Since $1\leq q<r$, one can easily see that $\frac{3r+1-2q}{r-q+1}\leq r+1$. An application of \cite[Lemma 15]{SSS} yields \eqref{383}. Using \eqref{382}, \eqref{383}, \eqref{368} and \eqref{371}, we have 
\begin{align}
&	\E^{\nu_0}\left[\left(f(X_t)-f(X_s)-\int_s^t\mathscr{L}f(X_r)\d r\right)f_0(X_{s_0})\cdots f_m(X_{s_m})\right]\nonumber\\&=\lim\limits_{n\to\infty}\E^{\nu_{\e_n}}\left[\left(f(X_t)-f(X_s)-\int_s^t\mathscr{L}f(X_r)\d r\right)f_0(X_{s_0})\cdots f_m(X_{s_m})\right]\nonumber\\&=\lim\limits_{n\to\infty}\E^{\nu_{\e_n}}\left[\left(f(X_t)-f(X_s)-\int_s^t\mathscr{L}^{\e_n}f(X_r)\d r\right)f_0(X_{s_0})\cdots f_m(X_{s_m})\right]\nonumber\\&=0. 
\end{align}
Therefore, $\M_{k,j}(t)$ defined in \eqref{380} is a martingale under $\nu_0$. The case of $q<r\in[1,2]$  can be established in a similar way. 

For $k\in\mathbb{N}$, let $g(\x)=(\x,\boldsymbol{e}_k)$, $\x\in\H$. By a similar argument, one can show that 
\begin{align}\label{391}
	\M_k(t):&=g(X_t)-g(\h)-\int_0^t\mathscr{L}g(X_r)\d r\nonumber\\&=(X_t,\boldsymbol{e}_k)-(\h,\boldsymbol{e}_k)-\mu\int_0^t(X_s,\A\boldsymbol{e}_k)\d s-\int_0^t\langle\B(X_s),\boldsymbol{e}_k\rangle\d s-\alpha\int_0^t\langle\widetilde{\mathcal{C}}(X_s),\boldsymbol{e}_k\rangle\d s\nonumber\\&\quad-\beta\int_0^t\langle\mathcal{C}(X_s),\boldsymbol{e}_k\rangle\d s-\int_0^t(\F(X_s),\boldsymbol{e}_k)\d s
\end{align}
is a martingale under $\nu_0$. 

\vskip 0.2 cm
{\bf S3.} An It\^o's formula together with \eqref{380} and \eqref{391} yield 
\begin{align}
	\langle\M_k,\M_j\rangle(t)=\int_0^t(\sigma(X_s),\boldsymbol{e}_k)(\sigma(X_s),\boldsymbol{e}_j)\d s,
\end{align}
where $\langle\M_k,\M_j\rangle$  represents  the sharp bracket of the two martingales. According to \cite[Lemma A.1]{SSTZ}, there exists a probability space $(\Omega',\mathscr{F}',\mathbb{P}')$ with a filtration $\mathscr{F}'_t$ such that on the standard extension $$(\Omega\times\Omega',\mathscr{F}\times\mathscr{F}',\mathscr{F}_t\times\mathscr{F}_t',\mathbb{P}\times\mathbb{P}')$$ of $(\Omega,\mathscr{F},\mathscr{F}_t,\mathbb{P})$, there exists a one-dimensional Brownian motion $\{\W_t\}_{t\geq 0}$ such that
\begin{align}
	\M_k(t)=\int_0^t(\sigma(X_s),\boldsymbol{e}_k)\d\W(s),
	\end{align}
which means 
\begin{align}
	(X_t,\boldsymbol{e}_k)-(\h,\boldsymbol{e}_k)&=-\mu\int_0^t(X_s,\A\boldsymbol{e}_k)\d s-\int_0^t\langle\B(X_s),\boldsymbol{e}_k\rangle\d s-\alpha\int_0^t\langle\widetilde{\mathcal{C}}(X_s),\boldsymbol{e}_k\rangle\d s\nonumber\\&\quad-\beta\int_0^t\langle\mathcal{C}(X_s),\boldsymbol{e}_k\rangle\d s+\int_0^t(\F(X_s),\boldsymbol{e}_k)\d s+\int_0^t(\sigma(X_s),\boldsymbol{e}_k)\d\W(s),
\end{align}
for every $k\geq 1$. Therefore, under $\nu_0$, $\{X_t\}_{t\geq 0}$ is a solution to the stochastic CBFeD equations \eqref{2p11}. By the uniqueness of the stochastic CBFeD equations, we conclude that $\nu_0=\nu$ completing the proof of the theorem.
\end{proof}
In the next theorem, we remove the restrictions on the noise coefficients and the initial value $\h$.

\begin{theorem}\label{thm3.11}
	Let Hypothesis \ref{hyp1}, \ref{hyp2} and \ref{hyp3} (that is, $(H.1)-(H.4)$) hold, and $\h\in\H$ be given. Then, for any $T > 0$, $\nu_{\e}$ converges	weakly to $\nu$, as $\e\to 0$, on the space $\D([0, T ];\H)$ equipped with the Skorohod topology. 
\end{theorem}
\begin{proof}
	For each $n\in\N$, let $\h_n$, $\F_n (\u)$, $\sigma_n (\u)$, $\sigma_n ^{\e}(\u, z)$  denote the corresponding orthogonal
	projections of $\h$, $\F(\u)$, $\sigma(\u)$, $\sigma^{\e}(\u, z)$ into the $n$-dimensional space $\mathrm{span}\{\boldsymbol{e}_1,\ldots,\boldsymbol{e}_n\}$. Then, for each $n\in\N$, $\{\sigma_n^{\e}\}_{0<\e\leq\e_0}$  and $\F_n$ satisfy Hypothesis $(H.2)-(H.5)$. Furthermore, there is a constant $C>0$ independent	of $n$ such that for every $\u, \u_1, \u_2 \in\H$ 
	\begin{align}\label{395}
		\sup_{n\in\N}\|\F_n(\u)\|_{\H}^2+\sup_{n\in\N}\|\sigma_n(\u)\|_{\H}^2+\sup_{n\in\N,0<\e\leq\e_0}\int_{\R_0}\|\sigma_n^{\e}(\u,z)\|_{\H}^2\lambda(\d z)\leq K_3(1+\|\u\|_{\H}^2)
	\end{align}
where $K_3=\max\{K,K_1\}$ and 
\begin{align}\label{396}
	&\sup_{n\in\N}\|\F_n(\u_1)-\F_n(\u_2)\|_{\H}^2+\sup_{n\in\N}\|\sigma_n(\u_1)-\sigma_n(\u_2)\|_{\H}^2\nonumber\\&\qquad+\sup_{n\in\N,0<\e\leq\e_0}\int_{\R_0}\|\sigma_n^{\e}(\u_1,z)-\sigma_n^{\e}(\u_2,z)\|_{\H}^2\lambda(\d z)\leq L_3\|\u_1-\u_2\|_{\H}^2,
\end{align}
where $L_3=\max\{L_1,L_2\}$. Let $\u_{n,\e}(\cdot)$ and $\u_{n}(\cdot)$ be the solutions of the following stochastic CBFeD equations in $\V'+\wi\L^{\frac{r+1}{r}}$, $\mathbb{P}$-a.s.: 
\begin{align}
	\u_{n,\e}(t)&=\h_n-\mu\int_0^t\A\u_{n,\e}(s)\d s-\int_0^t\B(\u_{n,\e}(s))\d s-\alpha\int_0^t\widetilde{\mathcal{C}}(\u_{n,\e}(s))\d s-\beta\int_0^t\mathcal{C}(\u_{n,\e}(s))\d s\nonumber\\&\quad+\int_0^t\F_n(\u_{n,\e}(s))\d s+\int_0^t\int_{\R_0}\sigma_n^{\e}(\u_{n,\e}(s-),z)\tilde{\pi}(\d s,\d z),\\
	\u_{n}(t)&=\h_n-\mu\int_0^t\A\u_{n}(s)\d s-\int_0^t\B(\u_{n}(s))\d s-\alpha\int_0^t\widetilde{\mathcal{C}}(\u_{n}(s))\d s-\beta\int_0^t\mathcal{C}(\u_{n}(s))\d s\nonumber\\&\quad+\int_0^t\F_n(\u_{n}(s))\d s+\int_0^t\sigma_n(\u_{n}(s))\d\W(s).
\end{align}
By Theorem \ref{thm310},  we have for each $n\in\N$, 
\begin{align}\label{399}
	\u_{n,\e}\to\u_n \ \text{ as }\ \e\to 0\ \text{ in distribution on the space }\  \D([0,T];\H). 
\end{align}
Furthermore, as in the proof of \eqref{3p12}, using \eqref{395}, one can show that 
\begin{align}
	&\sup\limits_{n\in\N,0<\e\leq\e_0} \E\left[\sup_{t\in[0,T]}\|\u_{n,\e}(t)\|_{\H}^{2p}+\left(\int_0^T\|\u_{n,\e}(t)\|_{\V}^2\d t\right)^p+\left(\int_0^T\|\u_{n,\e}(t)\|_{\wi\L^{r+1}}^{r+1}\d t\right)^p\right]<\infty,\label{3100}\\
&	\sup\limits_{n\in\N} \E\left[\sup_{t\in[0,T]}\|\u_{n}(t)\|_{\H}^{2p}+\left(\int_0^T\|\u_{n}(t)\|_{\V}^2\d t\right)^p+\left(\int_0^T\|\u_{n}(t)\|_{\wi\L^{r+1}}^{r+1}\d t\right)^p\right]<\infty,\label{3101}
	\end{align}
where $p=\frac{1+\e}{2}\max\{3,r+1\}$.  We claim that for any $\delta>0$,
\begin{align}
	\lim_{n\to\infty}\mathbb{P}\left\{\sup_{0\leq t\leq T}\|\u_n(t)-\u(t)\|_{\H}>\delta\right\}=0,\label{3102}\\
	\lim_{n\to\infty}\lim_{\e\to 0}\mathbb{P}\left\{\sup_{0\leq t\leq T}\|\u_{n,\e}(t)-\u_{\e}(t)\|_{\H}>\delta\right\}=0.\label{3103}
	\end{align}
We only prove \eqref{3103} here due to its similarity. 

\vskip 0.2 cm
\noindent\textbf{Case 1:} \emph{$d=2,3$ and $r\in(3,\infty)$. } 
Let us first consider the case $d=2,3$ and $r\in(3,\infty)$. An application of infinite dimensional  It\^o's formula to the process $\|\u_{n,\e}(\cdot)-\u_{\e}(\cdot)\|_{\H}^2$ yields  for all $t\in[0,T]$,
\begin{align}\label{3104}
&	\|\u_{n,\e}(t)-\u_{\e}(t)\|_{\H}^2+2\mu\int_0^t\|\u_{n,\e}(s)-\u_{\e}(s)\|_{\H}^2\d s\nonumber\\&\quad+2\beta\int_0^t\langle\mathcal{C}(\u_{n,\e}(s))-\mathcal{C}(\u_{\e}(s)),\u_{n,\e}(s)-\u_{\e}(s)\rangle\d s\nonumber\\&=\|\h_n-\h\|_{\H}^2-2\alpha\int_0^t\langle\widetilde{\mathcal{C}}(\u_{n,\e}(s))-\widetilde{\mathcal{C}}(\u_{\e}(s)),\u_{n,\e}(s)-\u_{\e}(s)\rangle\d s\nonumber\\&\quad-2\int_0^t\langle\B(\u_{n,\e}(s))-\B(\u_{\e}(s)),\u_{n,\e}(s)-\u_{\e}(s)\rangle\d s\nonumber\\&\quad+2\int_0^t\langle\F_n(\u_{n,\e}(s))-\F(\u_{\e}(s)),\u_{n,\e}(s)-\u_{\e}(s)\rangle\d s\nonumber\\&\quad+2\int_0^t\int_{\R_0}(\sigma_n^{\e}(\u_{n,\e}(s-),z)-\sigma^{\e}(\u_{\e}(s-),z),\u_{n,\e}(s-)-\u_{\e}(s-))\tilde{\pi}(\d s,\d z)\nonumber\\&\quad+\int_0^t\int_{\R_0}\|\sigma_n^{\e}(\u_{n,\e}(s),z)-\sigma^{\e}(\u_{\e}(s),z)\|_{\H}^2{\pi}(\d s,\d z), \ \mathbb{P}\text{-a.s.}
\nonumber\\:&=\sum\limits_{k=1}^6I^k_{n,\e}(t).
\end{align}
From \eqref{2.23}, we easily have 
\begin{align}\label{2.27}
	\beta	\langle\mathcal{C}(\u_{n,\e})-\mathcal{C}(\u_{\e}),\u_{n,\e}-\u_{\e}\rangle &\geq \frac{\beta}{2}\||\u_{n,\e}|^{\frac{r-1}{2}}(\u_{n,\e}-\u_{\e})\|_{\H}^2+\frac{\beta}{2}\||\u_{\e}|^{\frac{r-1}{2}}(\u_{n,\e}-\u_{\e})\|_{\H}^2.
\end{align}
Note that $\langle\B(\u_{n,\e},\u_{n,\e}-\u_{\e}),\u_{n,\e}-\u_{\e}\rangle=0$ and it implies that
\begin{equation}\label{441}
	\begin{aligned}
	&	\langle \B(\u_{n,\e})-\B(\u_{\e}),\u_{n,\e}-\u_{\e}\rangle \nonumber\\&=\langle\B(\u_{n,\e},\u_{n,\e}-\u_{\e}),\u_{n,\e}-\u_{\e}\rangle +\langle \B(\u_{n,\e}-\u_{\e},\u_{\e}),\u_{n,\e}-\u_{\e}\rangle \nonumber\\&=\langle\B(\u_{n,\e}-\u_{\e},\u_{\e}),\u_{n,\e}-\u_{\e}\rangle=-\langle\B(\u_{n,\e}-\u_{\e},\u_{n,\e}-\u_{\e}),\u_{\e}\rangle.
	\end{aligned}
\end{equation} 
Using H\"older's and Young's inequalities, we estimate $|\langle\B(\u_{n,\e}-\u_{\e},\u_{n,\e}-\u_{\e}),\u_{\e}\rangle|$ as  
\begin{align}\label{2p28}
	|\langle\B(\u_{n,\e}-\u_{\e},\u_{n,\e}-\u_{\e}),\u_{\e}\rangle|&\leq\|\u_{n,\e}-\u_{\e}\|_{\V}\|\u_{\e}(\u_{n,\e}-\u_{\e})\|_{\H}\nonumber\\&\leq\frac{\mu }{2}\|\u_{n,\e}-\u_{\e}\|_{\V}^2+\frac{1}{2\mu }\|\u_{\e}(\u_{n,\e}-\u_{\e})\|_{\H}^2.
\end{align}
We take the term $\|\u_{\e}(\u_{n,\e}-\u_{\e})\|_{\H}^2$ from \eqref{2p28} and use H\"older's and Young's inequalities to estimate it as 
\begin{align}\label{2.29}
	&\int_{\mathbb{T}^d}|\u_{\e}(x)|^2|\u_{n,\e}(x)-\u_{\e}(x)|^2\d x\nonumber\\&=\int_{\mathbb{T}^d}|\u_{\e}(x)|^2|\u_{n,\e}(x)-\u_{\e}(x)|^{\frac{4}{r-1}}|\u_{n,\e}(x)-\u_{\e}(x)|^{\frac{2(r-3)}{r-1}}\d x\nonumber\\&\leq\left(\int_{\mathbb{T}^d}|\u_{\e}(x)|^{r-1}|\u_{n,\e}(x)-\u_{\e}(x)|^2\d x\right)^{\frac{2}{r-1}}\left(\int_{\mathbb{T}^d}|\u_{n,\e}(x)-\u_{\e}(x)|^2\d x\right)^{\frac{r-3}{r-1}}\nonumber\\&\leq\frac{\beta\mu }{4}\left(\int_{\mathbb{T}^d}|\u_{\e}(x)|^{r-1}|\u_{n,\e}(x)-\u_{\e}(x)|^2\d x\right)+\eta\left(\int_{\mathbb{T}^d}|\u_{n,\e}(x)-\u_{\e}(x)|^2\d x\right),
\end{align}
for $r>3$, where $\eta= \frac{r-3}{r-1}\left(\frac{8}{\beta\mu (r-1)}\right)^{\frac{2}{r-3}}$. Using \eqref{2.29} in \eqref{2p28}, we find 
\begin{align}\label{2.30}
	&|\langle\B(\u_{n,\e}-\u_{\e},\u_{n,\e}-\u_{\e}),\u_{\e}\rangle|\nonumber\\&\leq\frac{\mu }{2}\|\u_{n,\e}-\u_{\e}\|_{\V}^2+\frac{\beta}{8}\||\u_{\e}|^{\frac{r-1}{2}}(\u_{n,\e}-\u_{\e})\|_{\H}^2+\frac{\eta}{2\mu}\|\u_{n,\e}-\u_{\e}\|_{\H}^2.
\end{align}
Let us now consider $\langle\widetilde{\mathcal{C}}(\u_{n,\e})-\widetilde{\mathcal{C}}(\u_{\e}),\u_{n,\e}-\u_{\e}\rangle $ and estimate it using Taylor's formula and H\"older's inequalities as
\begin{align}\label{3p100}
	&\langle\widetilde{\mathcal{C}}(\u_{n,\e})-\widetilde{\mathcal{C}}(\u_{\e}),\u_{n,\e}-\u_{\e}\rangle 
	\nonumber\\&=\int_{\mathbb{T}^d}\left(|\u_{n,\e}(x)|^{q-1}\u_{n,\e}(x)-|\u_{\e}(x)|^{q-1}\u_{\e}(x)\right)\cdot\w_{n,\e}(x)\d x\nonumber\\&=\int_{\mathbb{T}^d}|\u_{n,\e}(x)|^{q-1}|\w_{n,\e}(x)|^2\d x+\int_{\mathbb{T}^d}\left(|\u_{n,\e}(x)|^{q-1}-|\u_{\e}(x)|^{q-1}\right)\u_{\e}(x)\cdot\w_{n,\e}(x)\d x\nonumber\\&=\int_{\mathbb{T}^d}|\u_{n,\e}(x)|^{q-1}|\w_{n,\e}(x)|^2\d x\nonumber\\&\quad+\int_{\mathbb{T}^d}\int_0^1|\theta\u_{n,\e}(x)+(1-\theta)\u_{\e}(x)|^{q-3}\left(\u_{n,\e}(x)+(1-\theta)\u_{\e}(x)\right)\cdot\w_{n,\e}(x)\d\theta\left(\u_{\e}(x)\cdot\w_{n,\e}(x)\right)\d x\nonumber\\&\leq \int_{\mathbb{T}^d}|\u_{n,\e}(x)|^{q-1}|\w_{n,\e}(x)|^2\d x+2^{q-3}\int_{\mathbb{T}^d}\left(|\u_{n,\e}(x)|^{q-1}+|\u_{n,\e}(x)|^{q-2}|\u_{\e}(x)|\right)|\w_{n,\e}(x)|^2\d x\nonumber\\&\leq\left(1+2^{q-2}\right) \int_{\mathbb{T}^d}|\u_{n,\e}(x)|^{q-1}|\w_{n,\e}(x)|^{\frac{2(q-1)}{r-1}}|\w_{n,\e}(x)|^{\frac{2(r-q)}{r-1}}\d x\nonumber\\&\quad+\frac{(q-2)}{(q-1)^{\frac{q-1}{q-2}}}2^{\frac{(q-3)^2}{q-2}}\int_{\mathbb{T}^d}|\u_{\e}(x)|^{q-1}|\w_{n,\e}(x)|^{\frac{2(q-1)}{r-1}}|\w_{n,\e}(x)|^{\frac{2(r-q)}{r-1}}\d x\nonumber\\&\leq \frac{\beta}{4\alpha}\int_{\mathbb{T}^d}|\u_{n,\e}(x)|^{r-1}|\w_{n,\e}(x)|^{2}\d x+\frac{\beta}{8\alpha}\int_{\mathbb{T}^d}|\u_{\e}(x)|^{r-1}|\w_{n,\e}(x)|^{2}\d x+\chi\int_{\mathbb{T}^d}|\w_{n,\e}(x)|^{2}\d x,
\end{align}
for $q\in[3,r)$, where $\w_{n,\e}=\u_{n,\e}-\u_{\e}$ and $$\chi=\left(\frac{r-q}{r-1}\right)\left(\frac{4\alpha(q-1)}{\beta(r-1)}\right)^{\frac{q-1}{r-q}}\left[\left(1+2^{q-2}\right)^{\frac{r-1}{r-q}}+2^{\frac{q-1}{r-q}}\left(\frac{(q-2)}{(q-1)^{\frac{q-1}{q-2}}}2^{\frac{(q-3)^2}{q-2}}\right)^{\frac{r-1}{r-q}}\right].$$ The case of $1\leq q<3$ can be handled in a similar way. Making use of  \eqref{2.27}, \eqref{2.30} and \eqref{3p100} in \eqref{3104}, we deduce for all $t\in[0,T]$,
\begin{align}\label{3109}
	&	\|\u_{n,\e}(t)-\u_{\e}(t)\|_{\H}^2+\mu\int_0^t\|\u_{n,\e}(s)-\u_{\e}(s)\|_{\H}^2\d s+\frac{\beta}{2^{r}}\int_0^t\|\u_{n,\e}(s)-\u_{\e}(s)\|_{\wi\L^{r+1}}^{r+1}\d s\nonumber\\&\leq \|\h_n-\h\|_{\H}^2+\left(\frac{\eta}{\mu}+2\chi\right)\int_0^t\|\u_{n,\e}(s)-\u_{\e}(s)\|_{\H}^2\d s+I_{n,\e}^3(t)+I_{n,\e}^4(t)+I_{n,\e}^5(t), \ \mathbb{P}\text{-a.s.}
\end{align}
Using the Lipschitz continuity of $\F(\cdot)$, we estimate 
\begin{align}\label{3110}
	&\E\left[\sup_{s\in[0,t]}|I_{n,\e}^3(s)|\right]\nonumber\\&\leq 2\E\left[ \int_0^t\left(\|\F_n(\u_{n,\e}(s))-\F(\u_{n,\e}(s))\|_{\H}+\|\F(\u_{n,\e}(s))-\F(\u_{\e}(s))\|_{\H}\right)\|\u_{n,\e}(s)-\u_{\e}(s)\|_{\H}\d s\right]\nonumber\\&\leq C\E\left[\int_0^t\|\u_{n,\e}(s)-\u_{\e}(s)\|_{\H}^2\d s\right]+\E\left[\int_0^t\|\F_n(\u_{n,\e}(s))-\F(\u_{n,\e}(s))\|_{\H}^2\d s\right]. 
\end{align}
Applying  Brukholder-Davis-Gundy's and Young's inequalities, and Hypothesis \ref{hyp2},  we get 
\begin{align}\label{3111}
	&\E\left[\sup_{s\in[0,t]}|I_{n,\e}^4(s)|\right]\nonumber\\&\leq  2\sqrt{3}\E\left[\left(\int_0^t\int_{\R_0}\|\sigma_n^{\e}(\u_{n,\e}(s),z)-\sigma^{\e}(\u_{\e}(s),z)\|_{\H}^2\|\u_{n,\e}(s)-\u_{\e}(s)\|_{\H}^2{\pi}(\d s,\d z)\right)^{\frac{1}{2}}\right]\nonumber\\&\leq \frac{1}{2}\E\left[\sup_{s\in[0,t]}\|\u_{n,\e}(s)-\u_{\e}(s)\|_{\H}^2\right]+6\E\left[\int_0^t\int_{\R_0}\|\sigma_n^{\e}(\u_{n,\e}(s),z)-\sigma^{\e}(\u_{\e}(s),z)\|_{\H}^2\lambda(\d z)\d s\right]\nonumber\\&\leq \frac{1}{2}\E\left[\sup_{s\in[0,t]}\|\u_{n,\e}(s)-\u_{\e}(s)\|_{\H}^2\right]+12L_2\E\left[\int_0^t\|\u_{n,\e}(s)-\u_{\e}(s)\|_{\H}^2\d s\right]\nonumber\\&\quad+12\E\left[\int_0^t\int_{\R_0}\|\sigma_n^{\e}(\u_{n,\e}(s),z)-\sigma^{\e}(\u_{n,\e}(s),z)\|_{\H}^2\lambda(\d z)\d s\right].
\end{align}
A similar calculation as above gives 
\begin{align}\label{3112}
	\E\left[\sup_{s\in[0,t]}|I_{n,\e}^5(s)|\right]&\leq \E\left[\int_0^t\int_{\R_0}\|\sigma_n^{\e}(\u_{n,\e}(s),z)-\sigma^{\e}(\u_{\e}(s),z)\|_{\H}^2\lambda(\d z)\d s\right]\nonumber\\&\leq 2L_2\E\left[\int_0^t\|\u_{n,\e}(s)-\u_{\e}(s)\|_{\H}^2\d s\right]\nonumber\\&\quad+2\E\left[\int_0^t\int_{\R_0}\|\sigma_n^{\e}(\u_{n,\e}(s),z)-\sigma^{\e}(\u_{n,\e}(s),z)\|_{\H}^2\lambda(\d z)\d s\right].
	\end{align}
Combining \eqref{3110}-\eqref{3112} and substituting it in \eqref{3109}, we deduce 
\begin{align}
	&\E\left[\sup_{s\in[0,t]}	\|\u_{n,\e}(s)-\u_{\e}(s)\|_{\H}^2\right]+2\mu\E\left[\int_0^t\|\u_{n,\e}(s)-\u_{\e}(s)\|_{\H}^2\d s\right]\nonumber\\&\quad+\frac{\beta}{2^{r-1}}\E\left[\int_0^t\|\u_{n,\e}(s)-\u_{\e}(s)\|_{\wi\L^{r+1}}^{r+1}\d s\right]\nonumber\\&\leq 2\|\h_n-\h\|_{\H}^2+2\left(\frac{\eta}{\mu}+2\chi+14L_2\right)\E\left[\int_0^t\|\u_{n,\e}(s)-\u_{\e}(s)\|_{\H}^2\d s\right] \nonumber\\&\quad+2\E\left[\int_0^t\|\F_n(\u_{n,\e}(s))-\F(\u_{n,\e}(s))\|_{\H}^2\d s\right]\nonumber\\&\quad+28\E\left[\int_0^t\int_{\R_0}\|\sigma_n^{\e}(\u_{n,\e}(s),z)-\sigma^{\e}(\u_{n,\e}(s),z)\|_{\H}^2\lambda(\d z)\d s\right].
\end{align}
An application of Grownall's inequality yields
\begin{align}\label{3114}
	&\E\left[\sup_{s\in[0,t]}	\|\u_{n,\e}(s)-\u_{\e}(s)\|_{\H}^2\right]\nonumber\\&\leq\left\{2\|\h_n-\h\|_{\H}^2 +2\E\left[\int_0^t\|\F_n(\u_{n,\e}(s))-\F(\u_{n,\e}(s))\|_{\H}^2\d s\right]\right.\nonumber\\&\left.\qquad+28\E\left[\int_0^t\int_{\R_0}\|\sigma_n^{\e}(\u_{n,\e}(s),z)-\sigma^{\e}(\u_{n,\e}(s),z)\|_{\H}^2\lambda(\d z)\d s\right]\right\}e^{2\left(\frac{\eta}{\mu}+2\chi+14L_2\right)T}.
\end{align}
We claim that 
\begin{align}
	\lim_{n\to\infty}\lim_{\e\to 0}\E\left[\int_0^t\|\F_n(\u_{n,\e}(s))-\F(\u_{n,\e}(s))\|_{\H}^2\d s\right]=0,\label{3115}\\
	\lim_{n\to\infty}\lim_{\e\to 0}\E\left[\int_0^t\int_{\R_0}\|\sigma_n^{\e}(\u_{n,\e}(s),z)-\sigma^{\e}(\u_{n,\e}(s),z)\|_{\H}^2\lambda(\d z)\d s\right]=0. \label{3116}
\end{align}
If \eqref{3115} and \eqref{3116} hold true, then from \eqref{3114}, one can deduce that 
\begin{align}\label{3117}
\lim_{n\to\infty}\lim_{\e\to 0}	\E\left[\sup_{s\in[0,t]}	\|\u_{n,\e}(s)-\u_{\e}(s)\|_{\H}^2\right]=0,
\end{align}
and the claim \eqref{3103} follows. It is now only left to show \eqref{3116} as the proof of \eqref{3115} similar and simpler. Let us define 
\begin{align}
	\G_n^{\e}(\x):=\int_{\R_0}\|\sigma_n^{\e}(\x,z)-\sigma^{\e}(\x,z)\|_{\H}^2\lambda(\d z), \ \x\in\H. 
\end{align}
It should be noted that 
\begin{align}\label{3119}
	\sup_{n\in\N}\sup_{0<\e\leq\e_0}\G_n^{\e}(\x)\leq C(1+\|\x\|_{\H}^2). 
\end{align}
Using  \eqref{3p12} and the dominated convergence theorem, in order to prove \eqref{3100}, it suffices to show that for each $s\in [0, T ],$
\begin{align}\label{3120}
	\lim_{n\to\infty}\lim_{\e\to 0}\G_n^{\e}(\u_{n,\e}(s))=0. 
\end{align}

Confirming the three equalities set forth will yield \eqref{3120}.
\begin{align}
	&\lim_{\e\to 0}\G_n^{\e}(\u_{n,\e}(s))=\lim_{\e\to 0}\G_n^{\e}(\u_{n}(s)), \ \text{ for all }\ n\in\N,\label{3z1}\\
	&	\lim_{n\to\infty}\lim_{\e\to 0}\G_n^{\e}(\u_{n}(s))=\lim_{n\to\infty}\lim_{\e\to 0}\G_n^{\e}(\u(s)),\label{3z2}\\
	&\lim_{n\to\infty}\lim_{\e\to 0}\G_n^{\e}(\u(s))=0. \label{3z3}
\end{align}
Let us first prove \eqref{3z1}. Since $\u_{n}(\cdot)$ is a continuous process, due to \eqref{399}, we see that for each $n\in\mathbb{N}$, $s\in[0,T]$, 
\begin{align}
	\u_{n,\e}(s)\to\u_n (s)\ \text{ as }\ \e\to 0\ \text{ in distribution.}
\end{align}
In order to prove \eqref{3z1}, one can use Skorohod's representation theorem to assume that $\|\u_{n,\e}(s)-\u_{n}(s)\|_{\H}^2\to 0$, $\mathbb{P}$-a.s. as $\e\to 0$. In view of \eqref{3100}, $\left\{\|\u_{n,\e}\|_{\H}^2\right\}_{0<\e\leq\e_0}$  is uniformly integrable, and therefore, one can further deduce that
\begin{align}\label{312}
	\lim_{\e\to 0}\E\left[\|\u_{n,\e}(s)-\u_n(s)\|_{\H}^2\right]=0. 
\end{align}
On the other hand,
\begin{align}\label{312a}
	&\E\left[|\G_n^{\e}(\u_{n,\e}(s))-\G_n^{\e}(\u_{n}(s))|\right]\nonumber\\&
	\leq\E\left[\int_{\R_0}\left|\|\sigma_n^{\e}(\u_{n,\e}(s),z)-\sigma^{\e}(\u_{n,\e}(s),z)\|_{\H}^2-\|\sigma_n^{\e}(\u_{n}(s),z)-\sigma^{\e}(\u_{n}(s),z)\|_{\H}^2\right|\lambda(\d z)\right]\nonumber\\&\leq  \E\left[\int_{\R_0}\left|\|\sigma_n^{\e}(\u_{n,\e}(s),z)-\sigma^{\e}(\u_{n,\e}(s),z)\|_{\H}^2-\|\sigma_n^{\e}(\u_{n}(s),z)-\sigma^{\e}(\u_{n}(s),z)\|_{\H}^2\right|\lambda(\d z)\right]\nonumber\\&\leq  \E\left[\int_{\R_0}\left(\|\sigma_n^{\e}(\u_{n,\e}(s),z)-\sigma_n^{\e}(\u_{n}(s),z)\|_{\H}+\|\sigma^{\e}(\u_{n,\e}(s),z)-\sigma^{\e}(\u_{n}(s),z)\|_{\H}\right)\right.\nonumber\\&\left.\qquad\times\left( \|\sigma_n^{\e}(\u_{n,\e}(s),z)-\sigma^{\e}(\u_{n,\e}(s),z)\|_{\H}+\|\sigma_n^{\e}(\u_{n}(s),z)-\sigma^{\e}(\u_{n}(s),z)\|_{\H}\right)\lambda(\d z)\right]\nonumber\\&\leq\left\{\sqrt{2}\E\left[\int_{\R_0}\left(\|\sigma_n^{\e}(\u_{n,\e}(s),z)-\sigma_n^{\e}(\u_{n}(s),z)\|_{\H}^2+\|\sigma^{\e}(\u_{n,\e}(s),z)-\sigma^{\e}(\u_{n}(s),z)\|_{\H}^2\right)\lambda(\d z)\right]^{\frac{1}{2}}\right\}\nonumber\\&\quad\times 2\left\{\E\left[\int_{\R_0}\left(\|\sigma_n^{\e}(\u_{n,\e}(s),z)\|_{\H}^2+\|\sigma^{\e}(\u_{n,\e}(s),z)\|_{\H}^2+\|\sigma_n^{\e}(\u_{n}(s),z)\|_{\H}^2+\|\sigma^{\e}(\u_{n}(s),z)\|_{\H}^2\right)\lambda(\d z)\right]^{\frac{1}{2}}\right\}\nonumber\\&:=I_1^{\e}\times I_2^{\e}. 
\end{align}
Using \eqref{32}, \eqref{395}, \eqref{3100} and \eqref{3101}, we estimate 
\begin{align}\label{312b}
	\sup_{0<\e\leq\e_0}|I_2^{\e}|^2\leq C\sup_{n\in\N,\ 0<\e\leq\e_0}\E\left[1+\|\u_{n,\e}(s)\|_{\H}^2+\|\u_n(s)\|_{\H}^2\right]<\infty. 
\end{align}
Making use of \eqref{34}, \eqref{396} and \eqref{312}, one can estimate $|I_1^{\e}|^2$  as 
\begin{align}\label{312c}
	|I_1^{\e}|^2\leq C\E\left[\|\u_{n,\e}(s)-\u_n(s)\|_{\H}^2\right]\to 0\ \text{ as }\ \e\to 0. 
\end{align}
Therefore, \eqref{3z1} follows from \eqref{312a}, \eqref{312b} and \eqref{312c}. In view of \eqref{3102}, a similar argument as above leads to 
\begin{align}
	\lim_{n\to\infty}\sup_{0<\e\leq\e_0}	\E\left[|\G_n^{\e}(\u_{n}(s))-\G_n^{\e}(\u(s))|\right]=0,
\end{align}
so that \eqref{3z2} holds. Note that Hypothesis \ref{hyp3} (H.4) and the condition (ii) of (H.3) imply
\begin{align}\label{3130}
	&\lim_{n\to\infty}\lim_{\e\to 0}\int_{\R_0}\|\sigma_n^{\e}(\x,z)-\sigma^{\e}(\x,z)\|_{\H}^2\lambda(\d z)\nonumber\\&=\lim_{n\to\infty}\lim_{\e\to 0}\left[\int_{\R_0}\|\sigma^{\e}(\x,z)\|_{\H}^2\lambda(\d z)-\int_{\R_0}\|\sigma_n^{\e}(\x,z)\|_{\H}^2\lambda(\d z)\right]\nonumber\\&=\|\sigma(\x)\|_{\H}^2-\lim_{n\to\infty}\|\sigma_n(\x)\|_{\H}^2=0, \ \text{ for all }\ \x\in\H. 
\end{align}
Therefore, \eqref{3z3}  immediately follows from \eqref{3130} and \eqref{3119} by the dominated convergence theorem. Hence, \eqref{3116}  is proved, and so is \eqref{3117}. 

Finally, we prove that $\nu^{\e}$ converges weakly to $\nu$. Let $\nu^{\e}_n$ and $\nu^n$ denote the laws of $\u_{n,\e}$ and $\u_n$ on $\mathcal{S} := \D([0, T ];\H)$, respectively. Let $G$  be any given bounded, uniformly continuous function on $\mathcal{S}$. For any $n\geq 1$, we write
\begin{align}\label{3121}
	&	\int_{\mathcal{S}}G(\w)\nu^{\e}(\d\w)-\int_{\mathcal{S}}G(\w)\nu(\d\w)\nonumber\\&= 	\int_{\mathcal{S}}G(\w)\nu^{\e}(\d\w)-\int_{\mathcal{S}}G(\w)\nu^{\e}_n(\d\w)+\int_{\mathcal{S}}G(\w)\nu^{\e}_n(\d\w)-\int_{\mathcal{S}}G(\w)\nu_n(\d\w)\nonumber\\&\quad+\int_{\mathcal{S}}G(\w)\nu_n(\d\w)-\int_{\mathcal{S}}G(\w)\nu(\d\w)\nonumber\\&=\E\left[G(\u_{\e})-G(\u_{n,\e})\right]+\left(\int_{\mathcal{S}}G(\w)\nu^{\e}_n(\d\w)-\int_{\mathcal{S}}G(\w)\nu_n(\d\w)\right)+\E\left[G(\u_n)-G(\u)\right].
\end{align}
One can rewrite $\E\left[G(\u_{\e})-G(\u_{n,\e})\right]$ as 
\begin{align}\label{3122}
	\E\left[G(\u_{\e})-G(\u_{n,\e})\right]&=\E\left[G(\u_{\e})-G(\u_{n,\e});\sup_{0\leq t\leq T}\|\u_{n,\e}(t)-\u_{\e}(t)\|_{\H}\leq\delta_1\right]\nonumber\\&\quad+\E\left[G(\u_{\e})-G(\u_{n,\e});\sup_{0\leq t\leq T}\|\u_{n,\e}(t)-\u_{\e}(t)\|_{\H}>\delta_1\right],
\end{align}
for any $\delta_1>0$. Since $G$ is uniformly continuous, given any $\delta>0$, there exists a $\delta_1>0$ such that 
\begin{align}\label{3123}
	\left|\E\left[G(\u_{\e})-G(\u_{n,\e});\sup_{0\leq t\leq T}\|\u_{n,\e}(t)-\u_{\e}(t)\|_{\H}\leq\delta_1\right]\right|\leq\frac{\delta}{4},
\end{align}
for all $n\geq 1$ and $\e>0$. Using the fact that $G(\cdot)$ is bounded and taking in account of  \eqref{3103}, there exists an $n_1$  and then $\e_{n_1}$ such that
\begin{align}\label{3124}
	&	\sup_{0<\e\leq \e_{n_1}}\left|\E\left[G(\u_{\e})-G(\u_{n_1,\e});\sup_{0\leq t\leq T}\|\u_{n_1,\e}(t)-\u_{\e}(t)\|_{\H}>\delta_1\right]\right|\nonumber\\&\leq C\sup_{0<\e\leq \e_{n_1}}\mathbb{P}\left\{\sup_{0\leq t\leq T}\|\u_{n_1,\e}(t)-\u_{\e}(t)\|_{\H}\right\}\leq\frac{\delta}{4}. 
\end{align}
Using \eqref{3122} and \eqref{3123} in \eqref{3124}, we deduce
\begin{align}\label{3125}
	\left|\E\left[G(\u_{\e_{n_1}})-G(\u_{n_1,\e_{n_1}})\right]\right|\leq\frac{\delta}{2}. 
\end{align}
Since $G$ is bounded and uniformly continuous, using \eqref{3102}, we obtain 
\begin{align}\label{3126}
	\left|\E\left[G(\u_{n_1})-G(\u)\right]\right|\leq\frac{\delta}{4}. 
\end{align}
On the other hand, by using \eqref{399}, we obtain the existence of an $\e_1$ such that for $0<\e\leq \e_1$
\begin{align}\label{3127}
	\left|\int_{\mathcal{S}}G(\w)\nu^{\e}_{n_1}(\d\w)-\int_{\mathcal{S}}G(\w)\nu_{n_1}(\d\w)\right|\leq\frac{\delta}{4}.
\end{align}
Putting \eqref{3125}-\eqref{3127} together  in \eqref{3121}, we obtain for $\e\leq\min\{\e_{n_1},\e_1\}$ that 
\begin{align}
	\left|	\int_{\mathcal{S}}G(\w)\nu^{\e}(\d\w)-\int_{\mathcal{S}}G(\w)\nu(\d\w)\right|\leq \delta.
\end{align}
Since $\delta>0$ is arbitrarily small, we deduce 
\begin{align}
	\lim_{\e\to 0}\int_{\mathcal{S}}G(\w)\nu^{\e}(\d\w)=\int_{\mathcal{S}}G(\w)\nu(\d\w),
\end{align}
which completes the proof.

\vskip 0.2 cm
\noindent\textbf{Case 2:} \emph{$d=r=3$ and $2\beta\mu>1$. } 
For the case $d=3$ and $2\beta\mu> 1$, one needs to estimate the terms \eqref{2p28} and \eqref{3p100} only. Rest of the calculations follow as in the previous case. It can be easily seen that 
\begin{align}
|\langle\B(\u_{n,\e}-\u_{\e},\u_{n,\e}-\u_{\e}),\u_{\e}\rangle|&\leq\theta\mu\|\u_{n,\e}-\u_{\e}\|_{\V}^2+\frac{1}{4\theta\mu }\|\u_{\e}(\u_{n,\e}-\u_{\e})\|_{\H}^2,\\
|\langle\widetilde{\mathcal{C}}(\u_{n,\e})-\widetilde{\mathcal{C}}(\u_{\e}),\u_{n,\e}-\u_{\e}\rangle |&\leq\frac{\beta\theta}{\alpha}\|\u_{\e}(\u_{n,\e}-\u_{\e})\|_{\H}^2+\frac{\beta\theta}{\alpha}\|\u_{n,\e}(\u_{n,\e}-\u_{\e})\|_{\H}^2\nonumber\\&\quad+C\|\u_{n,\e}-\u_{\e}\|_{\H}^2,
\end{align}
for some $0<\theta<1$. 
\vskip 0.2 cm
\noindent\textbf{Case 3:} \emph{$d=2$ and $r\in[1,3]$. } 
For the case $d=2$ and $r\in[1,3]$, we apply infinite dimensional It\^o's formula to the process  $e^{-\gamma\int_0^{\cdot}\|\u_{\e}(s)\|_{\wi\L^4}^4\d s}\|\u_{n,\e}(\cdot)-\u_{\e}(\cdot)\|_{\H}^2$ to get 
\begin{align}\label{3p134}
&	e^{-\gamma\int_0^{t}\|\u_{\e}(s)\|_{\wi\L^4}^4\d s}\|\u_{n,\e}(t)-\u_{\e}(t)\|_{\H}^2 +2\mu\int_0^te^{-\gamma\int_0^{s}\|\u_{\e}(r)\|_{\wi\L^4}^4\d r}\|\u_{n,\e}(s)-\u_{\e}(s)\|_{\H}^2\d s\nonumber\\&\quad+2\beta\int_0^te^{-\gamma\int_0^{s}\|\u_{\e}(r)\|_{\wi\L^4}^4\d r}\langle\mathcal{C}(\u_{n,\e}(s))-\mathcal{C}(\u_{\e}(s)),\u_{n,\e}(s)-\u_{\e}(s)\rangle\d s\nonumber\\&=\|\h_n-\h\|_{\H}^2-\gamma\int_0^te^{-\gamma\int_0^{s}\|\u_{\e}(r)\|_{\wi\L^4}^4\d r}\|\u_{\e}(s)\|_{\wi\L^4}^4\|\u_{n,\e}(s)-\u_{\e}(s)\|_{\H}^2\d s\nonumber\\&\quad-2\alpha\int_0^te^{-\gamma\int_0^{s}\|\u_{\e}(r)\|_{\wi\L^4}^4\d r}\langle\widetilde{\mathcal{C}}(\u_{n,\e}(s))-\widetilde{\mathcal{C}}(\u_{\e}(s)),\u_{n,\e}(s)-\u_{\e}(s)\rangle\d s\nonumber\\&\quad-2\int_0^te^{-\gamma\int_0^{s}\|\u_{\e}(r)\|_{\wi\L^4}^4\d r}\langle\B(\u_{n,\e}(s))-\B(\u_{\e}(s)),\u_{n,\e}(s)-\u_{\e}(s)\rangle\d s\nonumber\\&\quad+2\int_0^te^{-\gamma\int_0^{s}\|\u_{\e}(r)\|_{\wi\L^4}^4\d r}\langle\F_n(\u_{n,\e}(s))-\F(\u_{\e}(s)),\u_{n,\e}(s)-\u_{\e}(s)\rangle\d s\nonumber\\&\quad+2\int_0^t\int_{\R_0}e^{-\gamma\int_0^{s}\|\u_{\e}(r)\|_{\wi\L^4}^4\d r}(\sigma_n^{\e}(\u_{n,\e}(s-),z)-\sigma^{\e}(\u_{\e}(s-),z),\u_{n,\e}(s-)-\u_{\e}(s-))\tilde{\pi}(\d s,\d z)\nonumber\\&\quad+\int_0^t\int_{\R_0}e^{-\gamma\int_0^{s}\|\u_{\e}(r)\|_{\wi\L^4}^4\d r}\|\sigma_n^{\e}(\u_{n,\e}(s),z)-\sigma^{\e}(\u_{\e}(s),z)\|_{\H}^2{\pi}(\d s,\d z), \ \mathbb{P}\text{-a.s.}
\end{align}
One can estimate the term $\langle\B(\u_{n,\e})-\B(\u_{\e}),\u_{n,\e}-\u_{\e}\rangle$ by using H\"older's, Ladyzheskaya's and Young's inequalities as 
\begin{align}\label{3p135}
	|\langle\B(\u_{n,\e})-\B(\u_{\e}),\u_{n,\e}-\u_{\e}\rangle|&=|\langle\B(\u_{n,\e}-\u_{\e},\u_{n,\e}-\u_{\e}),\u_{\e}\rangle|\nonumber\\&\leq\|\u_{\e}\|_{\wi\L^4}\|\u_{n,\e}-\u_{\e}\|_{\V}\|\u_{n,\e}-\u_{\e}\|_{\wi\L^4}\nonumber\\&\leq 2^{1/4}\|\u_{\e}\|_{\wi\L^4}\|\u_{n,\e}-\u_{\e}\|_{\V}^{3/2}\|\u_{n,\e}-\u_{\e}\|_{\H}^{1/4}\nonumber\\&\leq\frac{\mu}{2} \|\u_{n,\e}-\u_{\e}\|_{\V}^2+\frac{27}{32\mu^3}\|\u_{\e}\|_{\wi\L^4}^4\|\u_{n,\e}-\u_{\e}\|_{\H}^2. 
\end{align}
Using \eqref{2.27}, \eqref{3p100} and \eqref{3p135} in \eqref{3p134}, we obtain 
\begin{align}
	&	e^{-\gamma\int_0^{t}\|\u_{\e}(s)\|_{\wi\L^4}^4\d s}\|\u_{n,\e}(t)-\u_{\e}(t)\|_{\H}^2 +\mu\int_0^te^{-\gamma\int_0^{s}\|\u_{\e}(r)\|_{\wi\L^4}^4\d r}\|\u_{n,\e}(s)-\u_{\e}(s)\|_{\H}^2\d s\nonumber\\&\quad+\frac{\beta}{2^{r-1}}\int_0^te^{-\gamma\int_0^{s}\|\u_{\e}(r)\|_{\wi\L^4}^4\d r}\|\u_{n,\e}(s)-\u_{\e}(s)\|_{\wi\L^{r+1}}^{r+1}\d s\nonumber\\&\leq\|\h_n-\h\|_{\H}^2+2\int_0^te^{-\gamma\int_0^{s}\|\u_{\e}(r)\|_{\wi\L^4}^4\d r}\langle\F_n(\u_{n,\e}(s))-\F(\u_{\e}(s)),\u_{n,\e}(s)-\u_{\e}(s)\rangle\d s\nonumber\\&\quad+2\int_0^t\int_{\R_0}e^{-\gamma\int_0^{s}\|\u_{\e}(r)\|_{\wi\L^4}^4\d r}(\sigma_n^{\e}(\u_{n,\e}(s-),z)-\sigma^{\e}(\u_{\e}(s-),z),\u_{n,\e}(s-)-\u_{\e}(s-))\tilde{\pi}(\d s,\d z)\nonumber\\&\quad+\int_0^t\int_{\R_0}e^{-\gamma\int_0^{s}\|\u_{\e}(r)\|_{\wi\L^4}^4\d r}\|\sigma_n^{\e}(\u_{n,\e}(s),z)-\sigma^{\e}(\u_{\e}(s),z)\|_{\H}^2{\pi}(\d s,\d z), \ \mathbb{P}\text{-a.s.,}
\end{align}
by choosing $\gamma \geq \frac{27}{16\mu^3}$. The rest of the calculations can be completed as in the case of $d=2,3$ and $r\in(3,\infty)$. 
\end{proof}
	\begin{remark}
		Examples of $\{\sigma^{\e}\}_{\e>0}$ satisfying Hypotheses discussed in section \ref{Sec3} can be obtained from \cite[Section 4]{SSTZ}. 
	\end{remark}
	\medskip\noindent
	{\bf Acknowledgments:}    M. T. Mohan would  like to thank the Department of Science and Technology (DST) Science $\&$ Engineering Research Board (SERB), India for a MATRICS grant (MTR/2021/000066).

%	
%	\medskip\noindent	{\bf  Declaration:} 
%	
%	\noindent 	{\bf  Ethical Approval:}   Not applicable 
%	
%	\noindent  {\bf   Competing interests: } The authors declare no competing interests. 
%	
%%	\noindent 	{\bf   Authors' contributions: } All authors have contributed equally. 
%	
%	\noindent 	{\bf   Funding: } DST-SERB, India, MTR/2021/000066 (M. T. Mohan). 
%	
%	
%	
%	\noindent 	{\bf   Availability of data and materials: } Not applicable. 
%	

\end{document}